\documentclass[]{article}
\pdfoutput=1
\usepackage{amsmath}
\usepackage[scale=0.7]{geometry}
\usepackage{cases}
\usepackage{amssymb,amsfonts,amsthm}
\usepackage{bm}
\usepackage{bigstrut,multirow,rotating}
\usepackage{mathrsfs}
\usepackage{subfigure}
\makeatletter

\newcommand{\Rmnum}[1]{\expandafter\@slowromancap\romannumeral #1@}
\makeatother
\graphicspath{{arxiv/figs/}{figs/}}
\DeclareGraphicsExtensions{.pdf,.png,.jpg}
\usepackage{authblk}
\usepackage{hyperref}                                                  
\hypersetup{hypertex = true, colorlinks = true, linkcolor = blue, anchorcolor = blue, citecolor =blue}
\usepackage[square,numbers]{natbib}
\newtheorem{theorem}{Theorem}[section]
\newtheorem{proposition}[theorem]{Proposition}
\newtheorem{lemma}[theorem]{Lemma}

\newtheorem{definition}[theorem]{Definition}
\theoremstyle{remark}
\newtheorem{remark}[theorem]{Remark}

\newtheorem{assumption}{Assumption}
\numberwithin{equation}{section}
\newenvironment{shrinkeq}[1]%缩短公式之间的距离
{\bgroup
	\addtolength\abovedisplayshortskip{#1}
	\addtolength\abovedisplayskip{#1}
	\addtolength\belowdisplayshortskip{#1}
	\addtolength\belowdisplayskip{#1}}
{\egroup\ignorespacesafterend}
\usepackage{caption}
\usepackage{upgreek}
\usepackage{mathrsfs}
\usepackage{booktabs}
\usepackage{lipsum}
\usepackage{amsfonts}
\usepackage{graphicx}
\usepackage{epstopdf}
\usepackage{algorithm}
\usepackage{algorithmic}
\usepackage{color}
\usepackage{enumitem}
\usepackage{dsfont}

\newcommand{\abs}[1]{\left\vert#1\right\vert}
\newcommand{\set}[1]{\left\{#1\right\}}
\newcommand{\seq}[1]{\left(#1\right)}
\newcommand{\bra}[1]{\left[#1\right]}

\newcommand{\E}{\mathbb{E}}
\newcommand{\PP}[2]{\mathbb{P}_{#1}\left(#2\right)}
\newcommand{\SPP}[1]{\mathbb{P}\left(#1\right)}

\newcommand{\z}{\bm{z}}

\newcommand{\x}{\bm{x}}
\newcommand{\X}{\bm{X}}
\newcommand{\uu}{\bm{u}}
\newcommand{\vv}{\bm{v}}

\newcommand{\T}{\mathrm{T}}

\newcommand{\eexp}[1]{\exp\left(#1\right)}

\newcommand{\red}[1]{#1}

\title{The $L_p$-error rate for randomized quasi-Monte Carlo self-normalized importance sampling of unbounded integrands}

\author[1]{Jiarui Du}
\author[2]{Zhijian He \thanks{Corresponding author: hezhijian@scut.edu.cn}}

\affil[1,2]{School of Mathematics, South China University of Technology, Guangzhou 510641, Guangdong, People’s Republic of China}
\date{}
\begin{document}
\begin{sloppypar}
\maketitle
\begin{abstract}
Self-normalized importance sampling (SNIS) is a fundamental tool in Bayesian inference when the posterior distribution involves an unknown normalizing constant. In many applications, both the test function of interest and the underlying state space are unbounded, making direct $L_1$-error (mean absolute error) and $L_2$-error (root mean square error) estimates challenging for SNIS  under randomized quasi-Monte Carlo (RQMC) sampling.  
In this work, we derive the $L_p$-error rate $(p\ge1)$ for RQMC-based SNIS (RQMC-SNIS) estimators with unbounded integrands on unbounded domains. 
A key step in our analysis is to first establish the $L_p$-error rate for plain RQMC integration. Our results allow for a broader class of transport maps used to generate samples from RQMC points. Under mild function boundary growth conditions, we further establish the \(L_p\)-error rate of order \(\mathcal{O}(N^{-\beta + \epsilon})\) for RQMC-SNIS estimators, where $\epsilon>0$ is arbitrarily small, $N$ is the sample size, and \(\beta \in (0,1]\) depends on the boundary growth rate of the resulting integrand. Numerical experiments validate the theoretical results.
\end{abstract}

% REQUIRED
\par

% REQUIRED
\par

\section{Introduction}\label{intro}  
The Monte Carlo (MC) method is a standard approach for approximating expectations of the form $\pi(f) := \E_\pi[f(\bm{X})]$, where $\pi$ is the target distribution and $f$ is a test function. MC has been extensively applied in areas such as engineering, finance, and statistics \cite{glasserman2004}. Since MC requires sampling from $\pi$, situations where direct sampling is infeasible, most notably in Bayesian posterior inference, are typically addressed using importance sampling (IS). Given a tractable proposal distribution $q$, the standard IS estimator is defined as 
        % \begin{shrinkeq}{-2ex}
            \begin{equation}\label{eq:is}
            {I}_N^{\mathrm{IS}}(f) = \frac{1}{N} \sum_{i=1}^N \omega(\bm{x}_{i}) f(\bm{x}_{i}),
        \end{equation} where $N$ is  the sample size, $\bm{x}_{i} \sim q$, and \(\{\omega(\bm{x}_i)\}_{i=1}^N\) with \(\omega(\bm{x})=\pi(\bm{x})/q(\bm{x})\) are the importance weights.
        % \end{shrinkeq}
        
        However, in most practical Bayesian problems, the posterior density $\pi(\bm{x})$ is known only up to a normalizing constant, rendering the standard IS estimator infeasible. This issue can be addressed by the self-normalized importance sampling (SNIS) estimator, which normalizes the weights: 
        \begin{equation*}% E2-2: removed number (1.2); original label: eq:snis
            \pi_N(f) = \sum_{i=1}^N \frac{\omega(\bm{x}_{i})}{{\sum_{j=1}^N \omega(\bm{x}_{j})}} f(\bm{x}_{i}).
        \end{equation*}
        For bounded test functions, root mean square error (RMSE) bounds under independent and identically distributed (IID) sampling are well established \cite{sanz2023}. Agapiou et al. \cite{agapiou2017} extended these results by deriving RMSE bounds for unbounded test functions. Deligiannidis et al. \cite{deligiannidis2024} further generalized the analysis to arbitrary $L_p$-errors $(p\ge2)$ under suitable moment conditions on both the test function $f$ and the weight function $\omega$. In all cases, the RMSE converges at rate $\mathcal{O}(N^{-1/2})$, and the $L_p$-error established in \cite{deligiannidis2024} also achieves this rate, which may be slow for practical applications.

        The quasi-Monte Carlo (QMC) method approximates integrals by deterministic point
        sets rooted in uniform distribution theory, rather than IID random samples. For sufficiently regular integrands and suitably constructed point sets, QMC can achieve faster deterministic convergence rates than standard MC \cite{bookdick2010}.
        In practice, randomized QMC (RQMC) is more commonly used. 
        Suitable randomizations, such as scrambling or random shifts, retain the uniformity properties of the underlying point set while introducing randomness that permits error estimation via independent replications. 
        QMC and RQMC methods have been successfully applied in finance and statistics \cite{he2023,le2009,wang2017}. Given the advantages of RQMC over MC, several works have studied the incorporation of IS and SNIS with RQMC methods. For RQMC-based IS (RQMC-IS), Zhang et al. \cite{Zhang2021} investigated effective strategies for combining IS with conditional Monte Carlo and dimension-reduction techniques within the RQMC framework. Under the ``QMC-friendly" boundary growth conditions, He et al. \cite{he2023} proved the RMSE of RQMC-IS achieves a convergence rate of $\mathcal{O}(N^{-1+\epsilon})$, where $\epsilon>0$ is arbitrarily small, for unbounded functions defined on unbounded domains. The work \cite{he2023} also investigated the effect of different proposal distributions on the performance of RQMC-IS. The projection-based framework proposed by Ouyang et al. \cite{ouyang2024} provides an alternative approach for analyzing the RMSE convergence rate of RQMC-IS. 
        
        However, existing results for RQMC-based SNIS (RQMC-SNIS) do not cover the case of unbounded test functions on unbounded domains. 
        Dick et al. \cite{dick2019} derived an explicit error bound for the RQMC-SNIS estimator using the weighted star discrepancy. However, their result requires sufficiently regular conditions on the functions \(\omega\) and  \(f \), which are not satisfied when $f$ is unbounded. Gerber and Chopin \cite{sqmc2015} demonstrated the validity of RQMC-SNIS for unbounded test functions on bounded domains by introducing an extreme norm on \([0,1)^d\). However, extending this approach to unbounded domains is challenging. He et al. \cite{he2024} established an RMSE estimate for RQMC-SNIS under the assumption that the estimator $\pi_N(f)$ is uniformly bounded. To the best of our knowledge, for unbounded test functions on unbounded domains, even when the test functions are sufficiently smooth, there are currently no results on the $L_1$-error rate for the RQMC-SNIS estimator, let alone on the $L_p$-error rate for $p\ge 2$. In this work, we study the \( L_p \)-error rate for the RQMC-SNIS estimator allowing for unbounded test functions on unbounded domains. 

        A key step in establishing the \( L_p \)-error rate for MC-based SNIS \cite{deligiannidis2024} is to estimate the \( L_p \)-errors of $I_N^{\mathrm{IS}}(f)$ and $I_N^{\mathrm{IS}}(1)$ defined in \eqref{eq:is}. 
        For IID sampling, these estimates are relatively straightforward and can be derived
         using the Marcinkiewicz–Zygmund inequality \cite{ferger2014}. In contrast, for RQMC,
          where $\bm{x}_i = \mathcal{T}(\bm{u}_i) \sim q$ with RQMC points $\{\bm{u}_i\}_{i=1}^N$
           whose individual points are
        marginally distributed as \(\mathcal{U}(0,1)^d\) and a transport map $\mathcal{T}:(0,1)^d\to\mathbb{R}^d$, establishing the analogous $L_p$-error rate is substantially more challenging due to the inherent dependence among
        \(\bm u_1,\ldots,\bm u_N\). 
Existing results on the $L_p$-error rate for RQMC are limited to the inversion 
transformation $\mathcal{T} = \Phi^{-1}$, where $\Phi$ denotes the cumulative
 distribution function (CDF) of the standard normal distribution,
  and to test functions with sub-quadratic growth \cite{chen2025}.

Our first contribution is to extend the existing RQMC \(L_p\)-error
 results from the standard-normal inverse-CDF setting to a broader
 class of transport maps, and from sub-quadratic growth to the 
 critical quadratic-growth regime.
  Specifically, we show how the transport-map condition can be verified
   for several practically relevant transformations, including location-scale maps,
    component-wise nonlinear transformations, and composed transport maps.
     These verification results enlarge the class of proposal distributions
      available for the subsequent RQMC-SNIS analysis.

         Based on the established  $L_p$-error rate for plain RQMC integration, 
         our second contribution derives the \(L_p\)-error rate for RQMC-SNIS under
          suitable moment conditions. The resulting rate is of order
           \(\mathcal{O}(N^{-\beta + \epsilon})\), where \(\beta \in (0,1]\) 
           depends on the growth rate of $\omega f$. When $\omega f$ belongs to
            a broad class of “QMC-friendly” functions, $\beta$ approaches $1$, 
            yielding a nearly \(\mathcal{O}(N^{-1})\) convergence rate. 
            To the best of our knowledge, this is the first systematic
             $L_p$-error analysis of RQMC-SNIS for unbounded test functions
              on unbounded domains. Setting $p=1$ yields a mean absolute error rate,
               while $p=2$ corresponds to the RMSE rate. This contribution bridges
                a key theoretical gap in the analysis of RQMC-SNIS and substantially
                 broadens its applicability to practical Bayesian computation.
    
	    The paper is organized as follows. Section \ref{sec:back} provides a brief introduction to QMC and RQMC. Section~\ref{sec:QMC} derives the $L_p$-error rates for the RQMC and RQMC-SNIS estimators under the component-wise inverse-CDF setting.
        Section \ref{subsec:apply} extends the applicability of the results in Section~\ref{sec:QMC} to admissible composed transport maps.
        Section~\ref{sec:experiments} reports numerical experiments that validate the theoretical findings.  Conclusions are drawn in Section \ref{sec:conclusions}. Appendix~\ref{subsec:lemma} includes some auxiliary lemmas and technical proofs. Appendix~\ref{app:logistic} presents an additional numerical experiment on Bayesian logistic regression, and Appendix~\ref{appen:Lp} discusses applications of the $L_p$-error rate to skewness, kurtosis, and confidence intervals. 
        
\section{Background and notation}\label{sec:back}
    \subsection{Notation}\label{subsec:notation}
    In this work, all norms are taken to be Euclidean norms. Define $1/0:=\infty$ and $(s)!! = 1$ for integer $s\le 0$. We write \(\uppi\) for the circle constant, to distinguish it from the target distribution \(\pi\). To prevent confusion, we adopt the convention that bold symbols (e.g., $\bm{u}$, $\bm{v}$) denote vectors, while ordinary symbols denote scalars. For instance, $\bm{u}_j$ refers to the $j$-th point in a low-discrepancy point set, whereas $u_j$ denotes the $j$-th component of a $d$-dimensional vector $\bm{u}$. We write $1\!:\!d := \{1, 2, \ldots, d\}$. Throughout this paper, we denote by $\bm{v}$ a subset of $1\!:\!d$, $\bm{u}$ a point in $(0,1)^d$, and $\bm{x}$ a point in $\mathbb{R}^d$. The notation $\bm{a}_{\bm{v}}\!:\!\bm{b}_{\bar{\bm{v}}}$ means a vector in $\mathbb{R}^d$ whose $j$-th component is $a_j$ if $j\in\bm{v}$ and $b_j$ otherwise.  In particular, \(\boldsymbol u_{\bm v}\!:\!\mathbf 1_{\bar{\bm{v}}}\) means that the coordinates outside \(\bm v\) are fixed at the upper anchor \(1\).   Furthermore, $\partial^{\bm{v}} h(\bm{u})$ denotes the mixed partial derivative of $h(\bm{u})$ with respect to $u_j$ for all $j \in \bm{v}$. For $\bm{a}=\left(a_1, \ldots, a_d\right) \in \mathbb{N}_0^d$ where $\mathbb{N}_0^d$ is the set of all $d$-dimensional vectors of non-negative integers, define $|\bm{a}|=a_1+\cdots+a_d$, and
    $$
    D^{\bm{a}} f(\bm{x}):=\frac{\partial^{|\bm{a}|}}{\partial x_1^{a_1} \ldots \partial x_d^{a_d}} f(\bm{x}), \quad \forall \bm{x} \in \mathbb{R}^d.
    $$
     
     For functions $f$ and $g$, we write $f = \mathcal{O}(g)$ if
     there exists a constant $C>0$ such that $|f|\le C|g|$.   The upper bound \(\mathcal{O}(\cdot)\) is used to state a convergence rate as \(N\to\infty\), where the implicit constant is independent of the sample size \(N\), but may suffer from the curse of dimensionality. In estimates involving the auxiliary projection radius \(r\), constants are chosen independently of \(r\) unless explicitly stated otherwise. Finally, to distinguish between various sources of randomness, the expectations $\mathbb{E}_{\mathrm{RQMC}}$, $\mathbb E_q$ and \(\mathbb E_{\mathcal U}\)  are taken with respect to the randomization of the QMC point set, the distribution with density \(q\), and \(\bm u\sim \mathcal U(0,1)^d\) respectively. Similar notations apply for probabilities. Let $(\Omega_{\mathrm{RQMC}},\mathcal F_{\mathrm{RQMC}},
        \mathbb P_{\mathrm{RQMC}})$ denote the probability space on which the randomization of the QMC point set is defined.
    
	\subsection{Background on QMC integration}
		In this section, we provide a brief introduction to (R)QMC and review some recent results on the convergence rates of the RQMC method. Further details can be found in the monographs \cite{bookdick2022,bookdick2010,  booknieder1992}. 
		
		QMC is commonly used for the numerical evaluation of $d$-dimensional integral $\int_{(0,1)^d} h({\bm{u}}) \mathrm{d} {\bm{u}}$. The QMC estimator is \(\frac{1}{N} \sum_{i=1}^N h(\bm{u}_{i})\), where $\bm{u}_{i}\in (0,1)^d$. Unlike MC, which uses IID random samples, QMC uses deterministic point sets designed to achieve a high degree of uniformity over \((0,1)^d\). Such uniformity is frequently quantified by discrepancy. In this paper, we use the star discrepancy, defined by
		\begin{equation*}% E2-2: removed number (2.1); original label: eq:stardis
				D_N^*\left(\bm{u}_{1}, \ldots, \bm{u}_{N}\right) = \sup_{\bm{z} \in (0,1)^d} \left|\frac{1}{N} \sum_{i=1}^N \mathds{1}_{\left\{\bm{u}_{i} \in [\mathbf{0}, {\bm{z}})\right\}} - \prod_{j=1}^d z_j\right|.
			\end{equation*}
        Here \(\mathds{1}_{\{\mathcal{E}\}}\) denotes the indicator function of the event \(\mathcal{E}\). 
		A point set with a star discrepancy of $\mathcal{O}(N^{-1}(\log N)^{d-1})$ is referred to as a low-discrepancy point set. Commonly used low-discrepancy point sets include Sobol' point sets, Faure point sets, and others. 
        The classical Koksma-Hlawka inequality provides an error bound on the integration error:
		\begin{align}\label{eq:KH_inq}
				\left|\int_{(0,1)^d} h({\bm{u}}) \mathrm{d} {\bm{u}}-\frac{1}{N} \sum_{i=1}^N h(\bm{u}_{i})\right|\leq D_N^*\left(\bm{u}_{1}, \ldots, \bm{u}_{N}\right) V_{\mathrm{HK}}(h),
			\end{align}
        where \(V_{\mathrm{HK}}(h)\) denotes the variation of \(h\) in the sense of
Hardy and Krause. A function \(h\) is said to be of bounded variation in the sense of Hardy and
Krause (BVHK) if \(V_{\mathrm{HK}}(h)<\infty\). For completeness, we recall the definition of \(V_{\mathrm{HK}}(h)\); see Owen~\cite{owen2005hk} for further details.

A one-dimensional ladder on \([0,1]\) is a finite set
\[
    \mathcal Y^j
    =
    \{0=y_0^j<y_1^j<\cdots<y_{n_j}^j<1\}.
\]
The successor of \(y_k^j\in\mathcal Y^j\) is \(y_{k+1}^j\) if
\(k<n_j\), and is \(1\) if \(k=n_j\). A multidimensional ladder on
\([0,1]^m\) is a Cartesian product $ \mathcal Y
    =
    \mathcal Y^1\times\cdots\times\mathcal Y^m.$
For \(\bm y\in\mathcal Y\), its successor \(\bm y_+\) is defined
component-wise. The variation of \(\varphi\) defined over $[0,1]^m$ in the sense of Vitali is
\begin{equation}\label{eq:vit}
     V_{\mathrm{Vit}}(\varphi)
    =
    \sup_{\mathcal Y}\sum_{\bm y\in\mathcal Y}
    \left|
        \Delta(\varphi;\bm y,\bm y_+)
    \right|,
\end{equation}
where $\Delta(\varphi;\bm a,\bm b)=\sum_{\bm v\subseteq 1{:}m} (-1)^{|\bm v|}\varphi(\bm a_{\bm v}:\bm b_{\bar{\bm v}})$ denotes the $m$-fold alternating sum, and the supremum is taken over all multidimensional ladders on
\([0,1]^m\). For 
\(\emptyset\ne\bm v\subseteq1:d\), define $h_{\bm v}(\bm u_{\bm v})
    =
    h(\bm u_{\bm v}:\bm1_{\bar{\bm v}})$, $\bm u_{\bm v}\in[0,1]^{|\bm v|}.$
The Hardy--Krause variation of $h$ is
\begin{equation}\label{eq:hkvar}
     V_{\mathrm{HK}}(h)
    :=
    \sum_{\emptyset\ne\bm v\subseteq1:d}
    V_{\mathrm{Vit}}(h_{\bm v}).
\end{equation}
If \(\partial^{\bm 1:d}h\) exists on \([0,1]^d\), then \cite{owen2005hk} proved that
\begin{equation}\label{eq:hk}
    V_{\mathrm{HK}}(h)
    \le
    \sum_{\emptyset\ne\bm v\subseteq1:d}
    \int_{[0,1]^{|\bm v|}}
    \left|
    \partial^{\bm v}
   h(\bm u_{\bm v}:\bm1_{\bar{\bm v}})
    \right|\,d\bm u_{\bm v},
\end{equation}
where the equality holds if \(\partial^{\bm 1:d}h\) is continuous on \([0,1]^d\).
%
        
         % For its definition and properties, we refer to \cite{owen2005hk}.
        To facilitate error estimation, RQMC is commonly used. Typical randomization methods include random shifts, 
        % \cite{cranley1976},
        digital shifts, 
        % \cite{owenqmc} 
        and scrambling. 
        % \cite{owen1995}.
        To derive the convergence rate of the RQMC method, Owen \cite{owen2006} introduced a growth condition for unbounded functions defined on $(0,1)^d$ and established the $L_1$-error rate using a low-variation extension strategy. Building on this approach, He et al. \cite{he2023} further derived the $L_2$ convergence rate by exploiting the fact that the scrambled net variance is no worse than a constant times the MC variance.
     
        Recently, Ouyang et al. \cite{ouyang2024} refined the growth condition
         for unbounded functions on $\mathbb{R}^d$, focusing on the standard 
         normal target $\pi \sim \mathcal{N}(\bm{0},I_d)$, and established 
         the RMSE convergence rates for RQMC using a projection operator.
        Chen et al. \cite{chen2025} showed that the projection method can
         also be employed to derive the $L_p$-error rate,
          thereby extending the error analysis of RQMC with
           $\pi \sim \mathcal{N}(\bm{0},I_d)$.

	\section{The $L_p$-error rates of RQMC and RQMC-SNIS}\label{sec:QMC}

We focus on the $L_p$-error rate of the RQMC-SNIS estimator for $\pi(f)=\mathbb{E}_{\pi}[f(\bm X)]$ with an unbounded test function $f$ and unnormalized density $\bar \pi$. Let $q$ be a proposal density for IS satisfying $q(\bm x)>0$ whenever $\bar\pi(\bm x)>0$.
By a change of measure, we have
$$
        \pi(f) = \frac{\int_{\mathbb R^d} f(x)\bar\pi(x)\,dx}
       {\int_{\mathbb R^d} \bar\pi(x)\,dx} = \frac{\mathbb{E}_q[\omega(\bm X)f(\bm X)]}{\mathbb{E}_q[\omega(\bm X)]} := \frac{q(\omega f)}{q(\omega)},
$$
where $\omega= \bar\pi/ q$ and $q(g) := \int_{\mathbb{R}^d} g(\bm{x})q(\bm{x})d\bm{x}$. To obtain the $L_p$-error rate of the RQMC-SNIS estimator $\pi_{N}(f)$, it is critical to establish the $L_p$-error rates for estimating $q(g)$ with $g\in\{\omega f,\omega\}$. 
    
    \subsection{Assumptions and definitions} \label{subsec:ass}
Suppose that the proposal $q(\bm x)$ has independent marginal PDFs $q_j(x_j)$ for $j=1,\dots,d$, each of which is strictly positive and continuous on \(\mathbb R\). Assume that 
 \begin{equation}\label{eq:tau1}
 \bm X=\mathcal{T}(\bm u)=(\mathcal{T}_1(u_1),\dots,\mathcal{T}_d(u_d))\sim q,
 \end{equation}
where $\bm u=(u_1,\dots,u_d)\sim \mathcal{U}(0,1)^d$ and $\mathcal{T}_j$ is the inverse CDF for the distribution $q_j$. We use the extended-value convention $\mathcal{T}_j(0)=-\infty$ and $\mathcal{T}_j(1)=+\infty$.
\red{The independent-marginal structure is inherent
in the component-wise inverse-CDF construction,
while the positivity and continuity of the
marginal densities ensure the differentiability
of the inverse CDFs used in the subsequent
error analysis.}
More general proposals for $q$ are considered in Section \ref{subsec:apply}. Using the inversion method, the integral over $\mathbb{R}^d$ is transformed to an integral over $(0,1)^d$, given by
    $$
    q(g) =\int_{\mathbb{R}^d} g(\bm{x})q(\bm{x})d\bm{x} = \int_{(0,1)^d}g\circ \mathcal{T}(\bm{u})d \bm u,
    $$
and the RQMC estimator is given by
\begin{equation*}% E2-2: removed number (3.2); original label: eq:rqmc
    I_N(g)
    =
    \frac{1}{N}\sum_{i=1}^N g(\bm x_i)
    =
    \frac{1}{N}\sum_{i=1}^N g\circ \mathcal{T}(\bm u_i),
\end{equation*}
where \(\{\bm u_1,\ldots,\bm u_N\}\) is an RQMC point set in \((0,1)^d\). We thus rewrite the RQMC-SNIS estimator $\pi_N(f)$ as
\begin{equation}\label{eq:snisest}
        \pi_{N}(f) = \frac{\frac 1N\sum_{i=1}^{N}\omega\left(\bm{\bm{x}}_i\right)f\left(\bm{\bm{x}}_i\right) }{\frac 1N\sum_{j=1}^{N}\omega\left(\bm{\bm{x}}_j\right)}  = \frac{I_N(\omega f)}{I_N(\omega)}.
\end{equation}

To derive the $L_p$-error bound for RQMC integration $I_N(g)$, we need some technical conditions on the RQMC point set $\left\{\bm{u}_{1}, \ldots, \bm{u}_{N}\right\}$, the function $g$, and the proposal $q$ as outlined below in order. 

        \begin{assumption}\label{ass:rqmc}
            Suppose that $\left\{\bm{u}_{1}, \ldots, \bm{u}_{N}\right\}$ is an RQMC point set satisfying the following two properties: (i) Marginal uniformity: each \(\bm u_i\) has marginal distribution \(\mathcal U(0,1)^d\). (ii) Almost surely discrepancy bound:    
            \red{there} exists a constant $B_d > 0$, independent of \(N\), such that
        \begin{equation*}% E2-2: removed number (3.4); original label: eq:disc_bound
                \mathbb P_{\mathrm{RQMC}}\left[D_N^*(\bm{u}_1, \ldots, \bm{u}_N) \le B_d \frac{(\log N)^{d-1}}{N}\right]=1.
            \end{equation*}
        \end{assumption}
        \begin{remark}
        A standard example satisfying Assumption~\ref{ass:rqmc}, for sample sizes \(N=b^m\), is a digital \((t,m,d)\)-net in base \(b\), such as a Sobol' or Niederreiter point set, subjected to nested uniform scrambling proposed by Owen \cite{owen1995randomly}. See, for example,
\cite{bookdick2010,booknieder1992,owen1995randomly} for nested uniform scrambling, digital nets, and their discrepancy bounds.
        \end{remark}

         \begin{assumption}\label{ass:f}
         Suppose that \(g\in \red{C^d(\mathbb R^d)}\). Assume that there exist a growth rate \(M\in\mathbb R\) and a constant \(C>0\) such that, for every \(\bm a\in\mathbb N_0^d\) satisfying \(|\bm a|\le d\),
            $$
            \left|D^{\bm{a}} g(\x)\right| \le  C\exp \left\{M\|\x\|^2\right\}, \quad \forall \bm{x} \in \mathbb{R}^d.
            $$
        \end{assumption}

        \begin{remark}
            \red{
            The projection-based results in
\cite{chen2025,ouyang2024} primarily consider
sub-quadratic exponential growth rates $O(\exp \left\{M\|\x\|^k\right\})$ with $0<k<2$.
Although \cite[Section~5]{ouyang2024} also
considers the quadratic exponential growth $k=2$,
the resulting importance-weighted integrand
actually has a negative growth rate $M<0$ and
is therefore bounded.
In contrast, Assumption~\ref{ass:f} includes unbounded derivatives with quadratic exponential growth when $M>0$. Regarding regularity, the earlier
projection-based analyses impose conditions
on the first-order mixed partial derivatives,
whereas Assumption~\ref{ass:f} requires
$g\in C^d(\mathbb R^d)$ and bounds on all
$D^{\bm a}g$ with $|\bm a|\le d$.
The $C^d$ condition justifies the
multivariate Fa\`a di Bruno formula
\cite{cons1996} used in the radial-projection (Section~\ref{subsec:proj_op})
and transport-map analyses (Section~\ref{subsec:apply}), while the
full multi-index bounds control the
resulting derivative terms.
            }
              \end{remark}

        We now specify the tail conditions used below, which ensure that the tail probability of $q$ decays sufficiently rapidly to prevent the relevant expectations from diverging. 

        \begin{definition}\label{ass:x}
We say that a probability distribution $q$ satisfies a sub-Gaussian-type tail condition \cite{vershynin2018}, denoted by
\[
    q\in \mathrm{SG}(\eta,\alpha),
\]
if there exist constants \(\eta\in\mathbb Z\) and \(\alpha>0\) such that
$$
   \limsup_{t\to \infty} t^{-\eta}\exp(\alpha t^2) \mathbb P_{q}(\|\bm X\|>t)<\infty.
$$
\end{definition}

The following elementary tail integral bound will be used repeatedly.
        \begin{lemma}\label{lem:exp}
            For any \(s\in \mathbb{Z}\), constants $a>0$ and $t>1/\sqrt{2a}$, we have
            $$
            \int_t^{\infty} x^s \exp \left\{-a x^2\right\} d x \leq C_s(2a)^{-1}t^{s-1} \exp \left\{-a t^2\right\},
            $$
            where $C_s = \max\{(s+2)!!/2,1\}$.
        \end{lemma}
        \begin{proof}
            For $s\le 0$ and $x \ge t > 0$, $x^{s-1} \le t^{s-1}$, then
\[
\begin{aligned}
    \int_t^{\infty} x^s \exp\{-a x^2\}\,dx
    &=
    \int_t^{\infty} x^{s-1}x\exp\{-a x^2\}\,dx  \\
    &\le
    t^{s-1}\int_t^{\infty}x\exp\{-a x^2\}\,dx = \frac{1}{2a}t^{s-1}\exp\{-a t^2\}.
\end{aligned}
\]
For $s>0$, this follows from the same Gaussian-tail estimate as Lemma A.2 of \cite{ouyang2024}.
        \end{proof}

\begin{remark}\label{rem:gauss}
If $q$ is a standard Gaussian \(\mathcal N(\bm0,I_d)\), then
by the spherical-coordinate transformation and Lemma~\ref{lem:exp}, for \(t>1\),
\[
\begin{aligned}
    \mathbb P_{q}(\|\bm X\|>t)
    &=
    (2\red{\uppi})^{-d/2}2\red{\uppi}^{d/2}/\Gamma(d/2)
    \int_t^\infty \exp(-r^2/2)r^{d-1}\,dr  \\
    &\le
    2^{-d/2}(d+1)!!/\Gamma(d/2) t^{d-2}\exp(-t^2/2).
\end{aligned}
\]
Thus $q\in  \mathrm{SG}(d-2,1/2).$ 
\end{remark}

        \subsection{The projection operator for error analysis} \label{subsec:proj_op}  By the convention in Section \ref{subsec:notation}, the expectation $\mathbb E$ in this section is taken over the RQMC randomization.
        If $g$ satisfies Assumption~\ref{ass:f} with a negative-growth rate \(M<0\), 
        \(h(\bm u)=g\circ \mathcal{T}(\bm u)\) is of BVHK so that the Koksma-Hlawka \eqref{eq:KH_inq} is applicable (see Lemma~\ref{lem:negative_bvhk}). This is not the case for the nonnegative-growth rate $M\ge 0$.  We thus use a projection operator $P_r$ solely as a theoretical device in the error analysis, where $r>0$ is a tuning parameter. It is not involved in the actual RQMC or RQMC-SNIS estimators.

By the triangle inequality, we bound
\(\mathbb E|I_N(g)-q(g)|^p\) by the three terms
\begin{equation}\label{eq:err0}
    \E\big|{I}_{N}(g)-{I}_{N}(g\circ P_{r})\big|^{p}, ~~ \E\big|{I}_{N}(g\circ P_{r})-q(g\circ P_r)\big|^{p}~~ \text{and} ~~ {\big|q(g\circ P_r)-q(g)\big|}^{p}.
\end{equation}
Let $Y_i = g(\mathcal{T}(\bm u_i))-g(P_{r}(\mathcal{T}(\bm u_i))), i = 1,2,\ldots,N$. Although the RQMC points $\{\bm u_i\}_{i=1}^N$ satisfying Assumption \ref{ass:rqmc} are dependent, the function $|\cdot|^p$ is convex for $p \ge 1$. By Jensen's inequality and the linearity of expectation, we have
        \[ \E\big|{I}_{N}(g)-{I}_{N}(g\circ P_{r})\big|^{p} = \mathbb{E} \left| \frac{1}{N}\sum_{i=1}^N Y_i \right|^p \le \mathbb{E} \left( \frac{1}{N}\sum_{i=1}^N |Y_i|^p \right) = \frac{1}{N} \sum_{i=1}^N \mathbb{E}|Y_i|^p. \]
        Note that because each \(\bm {u}_i\sim \mathcal U(0,1)^d\)  marginally, the random variable \(\mathcal{T}(\bm {u}_i)\sim q\). Therefore, $\mathbb{E}|Y_i|^p = \mathbb{E}_{q} |g(\bm X) - g \circ P_r(\bm X)|^p$. Thus, we obtain
        \begin{equation}\label{eq:err1}
            \mathbb{E} |I_N(g) - I_N(g \circ P_r)|^p \le \mathbb{E}_q |g(\bm X) - g \circ P_r(\bm X)|^p.
        \end{equation}
        By Jensen's inequality again, it follows that
        \begin{equation}\label{eq:err2}
            {\big|q(g\circ P_r)-q(g)\big|}^{p} \le \E_{q}\left|g(\X)-g \circ P_{r}(\X)\right|^{p}.
        \end{equation}
        So it suffices to bound the  projection  error  $\E_{q}\left|g(\X)-g \circ P_{r}(\X)\right|^{p}$ and the quadrature error $\E\big|{I}_{N}(g\circ P_{r})-q(g\circ P_r)\big|^{p}$. We choose a suitable operator $P_r$ to make $g\circ P_{r}\circ \mathcal{T}(\bm u)$ BVHK and then choose $r$ depending on $N$ to minimize the $L_p$ error rate.

        We first recall the component-wise smoothed projection operator used in the
        earlier projection-based framework of~\cite{ouyang2024}. In one dimension, for \(r>1\), define \(P_r^{\mathrm{cw}}:\mathbb R\to\mathbb R\) by
        \begin{equation*}
         P_{r}^{\mathrm{cw}}(x) = 
        \begin{cases} 
        -r + \frac{1}{2}, & x \in (-\infty, -r], \\
        \frac{1}{2}x^2 + rx + \frac{(r-1)^2}{2}, & x \in (-r, -r + 1), \\
        x, & x \in [-r + 1, r - 1], \\
        -\frac{1}{2}x^2 + rx - \frac{(r-1)^2}{2}, & x \in (r - 1, r), \\
        r - \frac{1}{2}, & x \in [r, \infty).
        \end{cases}
        \end{equation*}
        For the multivariate case, the operator acts component-wise as 
        \begin{equation}\label{eq:ouyang}
            {P}_{r}^{\mathrm{cw}}(\boldsymbol{x}) = ({P}_{r}^{\mathrm{cw}}(x_1),\ldots,{P}_{r}^{\mathrm{cw}}(x_d)).
        \end{equation}
% BEGIN REVISION_1_ORIGINAL R1-07
% Source: revision_1/main_clean.tex:385--388
% Context: Section 3.2: component-wise and radial projection.
% Restore: replace the next 1 current source line(s) with the archived original.
% Original source follows; remove only its leading "% " comment prefix.
        Note that $ \|P_r^{\mathrm{cw}}(\bm x)\|
        \le
        \sqrt d\,r$ and $\exp\{M\|P_r^{\mathrm{cw}}(\bm x)\|^2\} \le
        \exp\{Mdr^2\}$ for \(M>0\), leading to an undesirable dimension factor
         in the convergence exponent, see Remark \ref{rem:necessity}
          for a detailed discussion. To remove this additional dimension
           factor from the convergence exponent, we construct a radially
            defined projection operator.
% END REVISION_1_ORIGINAL R1-07
        % \red{For $M>0$, the component-wise bound $\|P_r^{\mathrm{cw}}(\bm x)\|^2\le dr^2$ introduces $\exp(Mdr^2)$ into the projected derivative estimate. Our radial projection instead satisfies $\|P_r(\bm x)\|^2\le(1-\delta/2)^2r^2$, avoiding this additional dimension factor in the convergence exponent.}
        
        To ensure a smooth transition, we define
        \begin{equation*}% E2-2: removed number (3.9); original label: eq:psi
        \psi(t) = \frac{\int_0^t s^d(1-s)^d ds}{\int_0^1 s^d(1-s)^d ds},~0\le t\le1.
        \end{equation*}
        This is the CDF of a symmetric Beta distribution with parameter $d+1$. Clearly, $0\le\psi(t)\le 1$, $\psi(0)=0$, and $\psi(1)=1$. 
        \begin{lemma}\label{lem:psi}
        For $1\le k \le d$,
        $D^k \psi(0) = D^k\psi(1) = 0.$
        Moreover, 
        $$\max_{1\le k\le d}\sup_{0<t<1}|D^k\psi(t)|\le \frac{(2d+1)!}{(d!)^2}(2d)^{d-1}.$$
        \end{lemma}  
        
        \begin{proof}
            Write $\zeta(s) = s^d(1-s)^d$ and $c = 1/\int_0^1\zeta(s)ds = \frac{(2d+1)!}{(d!)^2}$, so that $\psi(t) = c\int_0^t\zeta(s)ds$. By the Leibniz rule, $D^k\psi(t) = cD^{k-1}\zeta(t),1\le k \le d$, where
            \begin{align*}
                D^{k-1}\zeta(t) = \sum_{i=0}^{k-1}\binom{k-1}{i}D^{i}(t^d)D^{k-1-i}((1-t)^{d}).
            \end{align*}
            Evaluating at the boundaries, for $1 \le k \le d$, we have 
            $$D^k\psi(0) = cD^{k-1}\zeta(0) = 0,\quad D^k\psi(1) = cD^{k-1}\zeta(1) = 0.$$
            Moreover, for any $0<t<1$ and $1\le k \le d$,
            $$
            |D^k\psi(t)| = |cD^{k-1}\zeta(t)| \le c\sum_{i=0}^{k-1}\binom{k-1}{i}d^{i}d^{k-1-i} = c  (2d)^{k-1} \le \frac{(2d+1)!}{(d!)^2}(2d)^{d-1}. 
            $$
            This completes the proof.
        \end{proof} 

        We now define the radial projection operator \(P_r:\mathbb R^d\to\mathbb R^d\) by
        \begin{equation}\label{eq:proj}
            P_{r}(\x)= \begin{cases}\x, & \|\x\| \leq (1-\delta )r, \\ (1-\psi(t(\bm x))) \x, & (1-\delta )r<\|\x\|<r, \\ \bm{0}, & \|\x\| \geq r,\end{cases} 
        \end{equation}
        where $(1-\delta )r > 1$, $t(\bm x) = (\|\x\|-(1-\delta )r)/(\delta r)$, and  $0< \delta < 1/2$ is a sufficiently small constant. Denote the $i$-th component of $P_{r}(\x)$ by $P_r^{(i)}(\bm x)$, $i = 1,2,\ldots, d$.

            \begin{figure}[htbp]
        \vspace{-4mm}
        \centering
        \subfigure[Our  projection operator \eqref{eq:proj}]{\label{pr_our}
        	\begin{minipage}[t]{0.4\linewidth}
        		\centering
        		\includegraphics[width=0.9\linewidth]{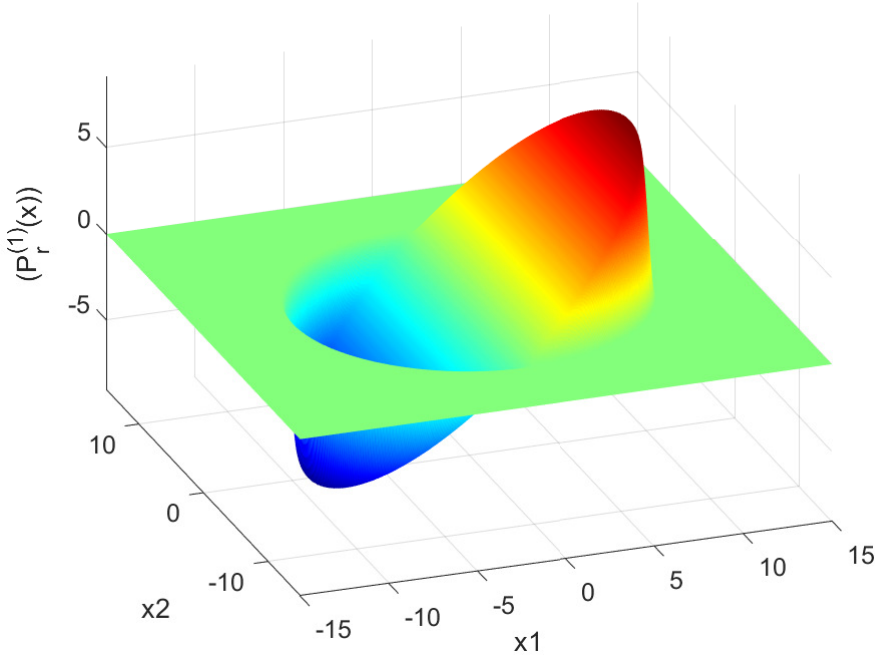}
        	\end{minipage}
        }
        % \vspace{-2mm}
        \subfigure[The projection operator \eqref{eq:ouyang} 
        % proposed by \cite{ouyang2024}
        ]{\label{pr_ouyang}
        	\begin{minipage}[t]{0.4\linewidth}
        		\centering
        		\includegraphics[width=0.9\linewidth]{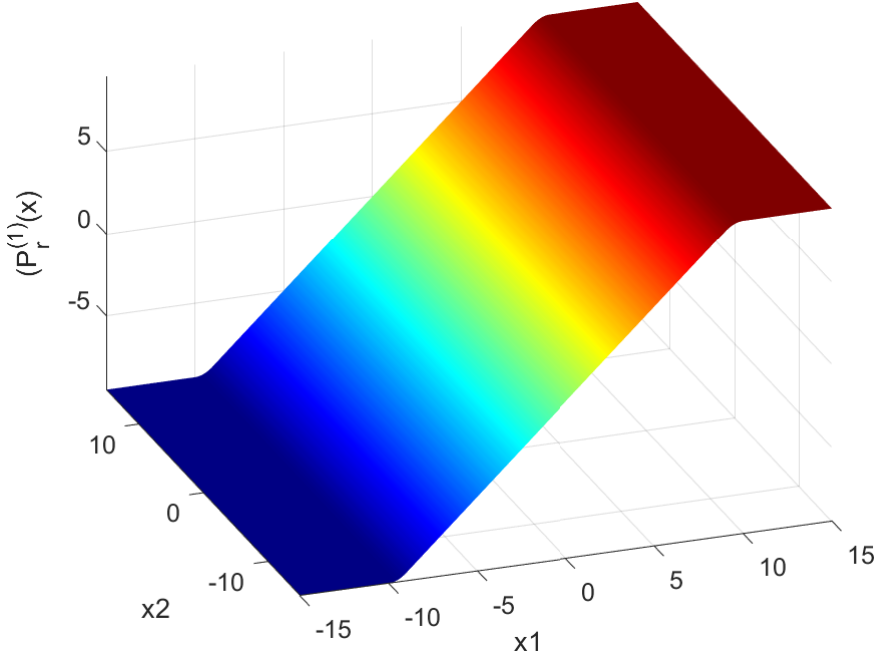}
        	\end{minipage}
        }\\
        \caption{A comparison of the first components of the two projection operators in $\mathbb{R}^2$ with $r=10$ for both operators and $\delta=0.1$ for our proposed operator.}
        \label{fig:pr}
        \vspace{-4mm}
        \end{figure}

         \begin{remark}
           A comparison of the two projection operators is given in Figure \ref{fig:pr}, where we illustrate the first component of the two-dimensional projection operator with $d=2,\ r=10,\  \delta = 0.1$. 
           It can be observed that, in the radial operator \eqref{eq:proj}, the function value of the operator depends on both components, whereas the component-wise operator \eqref{eq:ouyang} depends only on the first component. 
        \end{remark}

        \begin{lemma}
            \label{lem:pr}
             For any $\emptyset \neq \bm{v} \subseteq 1:d$ and $1\le i \le d$, the projection $P_{r}=(P_{r}^{(1)}(\x),\dots,P_{r}^{(d)}(\x))$   defined in \eqref{eq:proj} satisfies: 
            \begin{enumerate}[label=(\roman*)]
                \item \red{$P_r\in C^d(\mathbb R^d;\mathbb R^d)$, i.e., $P_r^{(i)}$ is $d$-times continuously differentiable on $\mathbb R^d$.}\label{pr1}
                % \item  $\partial^{\bm{v}} P_{r}^{(i)}(\x)$ exists for any $\bm{x}\in\mathbb{R}^d$,  \label{pr1}
                \item There exists a constant $C_P(d,\delta)$, independent of $r$, such that $\left|\partial^{\bm{v}} P_{r}^{(i)}(\x)\right| \le C_P(d,\delta) \mathds{1}_{\{\|\x\| < r\}}$, \label{pr2}
                \item For any $\x \in \mathbb{R}^d$, $\|P_{r}(\x)\| \le \|\x\|$ and $\|P_{r}(\x)\| \le (1-\delta /2)r$. \label{pr4}
            \end{enumerate}
        \end{lemma}
        % \end{theorem}
        
        \begin{proof}
        We prove the items in order.

        \emph{\ref{pr1}} In fact, the projection operator can be interpreted as a 
        vector-valued function whose components are scaled by radial factors.
         Specifically, $P_{r}^{(i)}(\x) = \xi(\rho) x_i$ where $\rho = \|\x\|$ and
        $$
        \xi(\rho)=
        \begin{cases}
        1, & \rho\le (1-\delta )r,\\[4pt]
        \kappa(\rho), & (1-\delta )r <\rho<r,\\[4pt]
        0, & \rho\ge r,
        \end{cases} \quad \kappa(\rho) = 1- \psi\left(t(\x)\right) =1-\psi\left(\dfrac{\rho-(1-\delta )r}{\delta r}\right).
        $$
        It is straightforward to verify that the smoothness of the 
        projection operator is completely determined by the smoothness of
         $\xi(\rho)$. By Lemma \ref{lem:psi}, we check that $\xi(\rho)$ 
         and its derivatives up to order $d$ match continuously 
         at $\rho = (1-\delta )r$ and $\rho=r$. 
        %  Hence, for any $i = 1,2,\ldots,d$ and $\emptyset \neq \bm{v} \subseteq 1:d$ ,
        %   $\partial^{\bm{v}} P_{r}^{(i)}(\x)$ exists for any $\bm{x}\in \mathbb{R}^d$.
        \red{Since $\rho=\|\x\|$ is smooth away from the
origin and $P_r(\x)=\x$ in a neighborhood
of the origin, we conclude that
$P_r\in C^d(\mathbb R^d;\mathbb R^d)$.}

        \emph{\ref{pr2}} We begin by considering the region $(1-\delta )r < \rho < r$. 
        We first record an estimate for derivatives of $\kappa(\rho)$ with respect to
         $\bm x$. Clearly, $\abs{\partial^{\bm{v}}\kappa(\rho)} = 
         \abs{\partial^{\bm{v}}\left(\psi(t(\bm x))\right)}$. 
         \red{On this region, the map $t(\x)$ is smooth and
$\psi$ is a polynomial.}
         By a special form of the Faa di Bruno formula \cite{basu2016}, 
        $$
\partial^{\bm{v}} \psi(t(\x))
=
\sum_{m=1}^{|\bm v|}
\psi^{(m)}(t(\x))
\sum_{\{\bm\ell_1,\ldots,\bm\ell_m\} \in \widetilde{\mathrm{KL}}(m,\bm{v})}
\prod_{j=1}^m
\partial^{\bm{\ell}_j}t(\x),
        $$
% BEGIN REVISION_1_ORIGINAL R1-08
% Source: revision_1/main_clean.tex:486--490
% Context: Lemma 3.8: original partition definition.
% Restore: replace the next 6 current source line(s) with the archived original.
% Original source follows; remove only its leading "% " comment prefix.
%         where \begin{equation}\label{eq:KLs}
%                 \widetilde{\mathrm{KL}}(m,\bm{v}) = 
%             \left\{\left(\bm{\ell}_1,\ldots,\bm{\ell}_m\right)\mid \bm{\ell}_j \subseteq \bm{v}, \cup_{j=1}^{m} \bm{\ell}_j=\bm{v},\bm{\ell}_j \cap \bm{\ell}_{j^{\prime}} =\emptyset \text { for } j \neq j^{\prime}
%             \right\} .
%             \end{equation}  
% END REVISION_1_ORIGINAL R1-08
        where \begin{equation*}% E2-2: removed number (3.11); original label: eq:KLs
                \widetilde{\mathrm{KL}}(m,\bm{v}) = 
            \left\{\red{\{}\bm{\ell}_1,\ldots,\bm{\ell}_m \red{\}}\mid \red{\emptyset\ne}\bm{\ell}_j \subseteq \bm{v}, \cup_{j=1}^{m} \bm{\ell}_j=\bm{v},\bm{\ell}_j \cap \bm{\ell}_{j^{\prime}} =\emptyset \text { for } j \neq j^{\prime}
            \right\} .
            \end{equation*}  
By the definition of $t(\x)$, for every nonempty $\bm \ell\subseteq 1:d$, we have
    \begin{equation*}% E2-2: removed number (3.12); original label: Dx
        |\partial^{\bm{\ell}}t(\x)| = \left|\partial^{\bm{\ell}} \left[\frac{\rho-(1-\delta )r}{\delta r}\right]\right| = \frac{|\partial^{\bm{\ell}}\rho|}{\delta r}  \le \frac{b_{|\bm \ell|}\rho^{1-|\bm \ell|}}{\delta r}. 
    \end{equation*}
       where we use the fact that for $\rho >0$, $ |\partial^{\bm \ell} \rho| \le
    b_{|\bm \ell|}\rho^{1-|\bm \ell|}, b_i=(2i-3)!!,\ i\ge1.$
        Then, 
        $$
        \left|\prod_{j=1}^m
\partial^{\bm{\ell}_j}t(\x)\right| \le \prod_{j=1}^m\frac{b_{|\bm \ell_j|}\rho^{1-|\bm \ell_j|}}{\delta r} \le r^{-1}\prod_{j=1}^{m}(b_{|\bm \ell_j|}/\delta) \le (b_{d}/\delta)^dr^{-1},
        $$
        where we used the fact that $\rho > (1-\delta )r > 1$ and $r > 1$. By Lemma \ref{lem:psi}, we have
        \begin{align*}
            \abs{\partial^{\bm{v}}\kappa(\rho)} = \abs{\partial^{\bm{v}}\left(\psi(t(\bm x))\right)} \le   C_{\kappa}(d,\delta)r^{-1}
        \end{align*}
        for a constant $C_{\kappa}(d,\delta)>0$.
        For $|\bm v| >1$, since $P_{r}^{(i)}(\x) = \kappa(\rho)x_i$, we have 
            $$
            \begin{aligned}
            |\partial^{\bm{v}} P_{r}^{(i)}(\x) |&\le  \left| \partial^{\bm{v}}\kappa(\rho)x_i\right|+\mathds{1}_{\{i \in \bm{v}\}} \left|\partial^{\bm{v}-\{{i}\}}\kappa(\rho)\right|\\
            &\le C_{\kappa}(d,\delta)|x_i|/r+\mathds{1}_{\{i \in \bm{v}\}}C_{\kappa}(d,\delta)/r\le 2C_{\kappa}(d,\delta),
             \end{aligned}
            $$
            where we used the fact that $|x_i| \le \rho \le r$. Similarly, if $|\bm{v}| =1$,
            \[ |\partial^{\bm{v}} P_{r}^{(i)}(\x)| \le |\partial^{\bm{v}}\kappa(\rho)||x_i| + \mathds{1}_{\{i \in \bm{v}\}}|\kappa(\rho)| \le C_{\kappa}(d,\delta)\frac{|x_i|}{r} + \abs{1-\psi(t)} \le C_{\kappa}(d,\delta)+1, \]
            where we used the fact that $0\le \psi(t) \le 1$.
            For $\rho \le (1-\delta )r$ and $\rho \ge r$, it is clear that  
            $
            \left|\partial^{\bm{v}} P_{r}^{(i)}(\x)\right| \le \mathds{1}_{\{\rho \le (1-\delta )r\}}.
            $
            Combining these results gives $|\partial^vP_r^{(i)}(x)|
        \le
        C_P(d,\delta)\mathbf 1_{\{\|x\|<r\}}$, where $C_P(d,\delta) = 2C_{\kappa}(d,\delta)+1$.

            \emph{\ref{pr4}}             
            It is straightforward to check that for $\rho \le (1-\delta )r$ and $\rho \ge r$, $\|P_{r}(\x)\| \le (1-\delta /2)r$. We focus on $\rho \in ((1-\delta )r, r)$. Note that because $t:=t(\bm x) = \frac{\rho-(1-\delta )r}{\delta r}$, for $0\le t \le 1/2$,  
             $$\|P_{r}(\x)\| = (1-\psi(t))\rho \le \rho = \delta rt+(1-\delta )r \le (1-\delta /2)r.$$
             In addition, note that because $\zeta(s) = s^d(1-s)^d$ is symmetric about $s = 1/2$, we have 
             $$
             \int_{0}^{1/2}\zeta(s)ds = \int_{1/2}^1\zeta(s)ds = \frac{1}{2}\int_{0}^{1}\zeta(s)ds.
             $$
             Hence, for $1/2 < t \le 1$, 
             $$
             \psi(t) = \frac{\int_0^t s^d(1-s)^d ds}{\int_0^1 s^d(1-s)^d ds} > \frac{\int_0^{1/2} s^d(1-s)^d ds}{\int_0^1 s^d(1-s)^d ds} = \frac{1}{2}.
             $$
             Therefore, 
             $$\|P_{r}(\x)\| = (1-\psi(t))\rho < \rho/2 \le r/2 < (1-\delta /2)r,$$
             completing the proof.
        \end{proof}

        \subsection{Main results and implications} \label{subsec:main}
        We now state the main \(L_p\)-error result for plain RQMC integration in Theorem \ref{thm:rqmc}. This main theorem separates two regimes whose proofs are deferred to Appendix~\ref{subsec:lemma}. If \(M<0\), then the transformed integrand \(h(\bm u)=g\circ \mathcal{T}(\bm u)\) is already of BVHK, and no projection operator or tail condition is needed. If \(M\ge0\), the sub-Gaussian-type tail condition and the radial projection operator are used to control the projection and quadrature errors.

        \begin{theorem}\label{thm:rqmc}
    Let \(p\ge1\). Suppose that \(g\) satisfies Assumption~\ref{ass:f} with a growth rate \(M\in\mathbb R\). Suppose that the inversion transformation \eqref{eq:tau1} is used in the RQMC quadrature with the RQMC point set satisfying Assumption~\ref{ass:rqmc}.\\
\(\mathrm{(i)}\) If \(M<0\), then 
\[
    \left\{
    \mathbb E_{\mathrm{RQMC}}
    \left|I_N(g)-q(g)\right|^p
    \right\}^{1/p}
    =
    \mathcal O\left(N^{-1}(\log N)^{d-1}\right).
\]
\(\mathrm{(ii)}\)  If $q\in \mathrm{SG}(\eta,\alpha)$ with constants \(\eta\in\mathbb Z\) and \(\alpha>0\), and \(0\le M<\alpha/p\), then for any \(\epsilon>0\),
\[
    \left\{
    \mathbb E_{\mathrm{RQMC}}
    \left|I_N(g)-q(g)\right|^p
    \right\}^{1/p}
    =
    \mathcal O\left(N^{-1+pM/\alpha+\epsilon}\right).
\] 
\end{theorem}

Theorem \ref{thm:rqmc} shows how the RQMC rate is determined by the competition between the growth rate \(M\) and the sub-Gaussian tail parameter \(\alpha\). If \(M<0\), the first part of Theorem~\ref{thm:rqmc} gives the nearly $\mathcal O\left(N^{-1}\right)$ rate without imposing a tail condition on the proposal $q$.
 If \(0 \le M<\alpha/p\), then the rate becomes nearly $\mathcal{O}\left(N^{-1+pM/\alpha}\right)$. In particular, the condition \(M<\alpha/(2p)\) guarantees a rate strictly faster than the Monte Carlo rate \(\mathcal{O}(N^{-1/2})\). Heavy-tailed proposals may cause the two integrands  \(\omega\) and
\(\omega f\) to have negative quadratic growth rates. Hence the regime \(M<0\) in
Theorem~\ref{thm:rqmc}, which does not require a tail condition on the proposal $q$, 
is particularly relevant to  the SNIS analysis in the first part of Theorem \ref{thm:SNIS}.
%  In the second part of Theorem \ref{thm:SNIS}, the importance weight $\omega$ 
%  has a mild growth rate $M_{\omega}$ in the sense that $M_{\omega}>0$ can be arbitrarily small.
%   We do not study the situation that $M_{\omega}>0$ cannot be arbitrarily small
%    since the proposal $q$ would make the two integrands severe for QMC. 
% Such a proposal $q$ is thus not a good candidate for QMC.

\begin{lemma}\label{lem:mis}
    Assume that $g_1$ and $g_2$ satisfy Assumption \ref{ass:f} with growth rates $M_1$ and $M_2$, respectively. Then $g=g_1g_2$  satisfies Assumption~\ref{ass:f} with a growth rate $M_{g} = M_1+M_2$.
\end{lemma}

\begin{proof}
    By the Leibniz rule \cite{cons1996}, for any $\bm{a} \in \mathbb{N}_0^d$ satisfying $|\bm{a}| \le d$, we have
    \begin{shrinkeq}{-0ex}
        \[
    \begin{aligned}
    |D^{\boldsymbol{a}}g(\bm{x})|&= \left| \sum_{\bm{0} \leq \boldsymbol{b} \leq \boldsymbol{a}}\binom{\boldsymbol{a}}{\bm{b}} D^{\boldsymbol{b}} g_1(\x) D^{\boldsymbol{a-b}} g_2(\x) \right| \\
    &\le \sum_{\bm{0} \leq \boldsymbol{b} \leq \boldsymbol{a}}\binom{\boldsymbol{a}}{\boldsymbol{b}} C_1 \exp\left(M_{1}\|\x\|^2\right) C_2 \exp\left(M_2\|\x\|^2\right) \\
    &= C_1 C_2 \left( \sum_{\bm{0} \leq \boldsymbol{b} \leq \boldsymbol{a}}\binom{\boldsymbol{a}}{\boldsymbol{b}} \right) \exp\left((M_1+M_2)\|\x\|^2\right)\\
    &\le 2^dC_1 C_2 \exp\left((M_1+M_2)\|\x\|^2\right),
    \end{aligned}
    \]
    \end{shrinkeq}
    where $\bm{0} \le \bm{b}\le \bm{a}$ means $0 \le b_i \le a_i$ for $i = 1,2,\ldots,d$, and the fact $\sum_{\bm{0} \leq \boldsymbol{b} \leq \boldsymbol{a}}\binom{\boldsymbol{a}}{\boldsymbol{b}} = 2^{|\bm{a}|} \le 2^d$ has been used. Therefore, $g$ satisfies Assumption \ref{ass:f} with a growth rate $M_{g} = M_1+M_{2}$.
\end{proof}

\begin{theorem}\label{thm:SNIS}
Suppose that the inversion transformation \eqref{eq:tau1} is used in the RQMC-SNIS estimator \eqref{eq:snisest} with the RQMC point set satisfying Assumption~\ref{ass:rqmc}. Assume that $\omega$ and $f$ satisfy Assumption~\ref{ass:f} with growth rates \(M_{\omega}\in\mathbb{R}\) and \(M_{f}\in\mathbb{R}\), respectively.\\
    \(\mathrm{(i)}\) If \(\max\{M_{\omega},M_{\omega}+M_f\}<0\), then for any $p\ge 1$, 
    $$
    \left\{
    \E_{\mathrm{RQMC}}
    \left|\pi_{N}(f)-\pi(f)\right|^p
    \right\}^{1/p}
    =
    \mathcal O\left(N^{-1}(\log N)^{d-1}\right).
    $$
\(\mathrm{(ii)}\) Suppose that there exists a constant $v> 1$ such that $q\left(|f|^v\right) = \int_{\mathbb{R}^d} |f(\bm{x})|^vq(\bm{x})d\bm{x}<\infty$. Let $1\le p <v$.  If \(M_{\omega}> 0\) is arbitrarily small, $q\in \mathrm{SG}(\eta,\alpha)$ with $\alpha>0$, and \(0\le M_f<\alpha/p\), then for any $\epsilon>0$,
    \[
    \left\{\mathbb{E}_{\mathrm{RQMC}}\left|\pi_{N}(f)-\pi(f)\right|^{p}\right\}^{1/p} = \mathcal{O}(N^{-1+p M_f/\alpha+\epsilon}).
    \]
\end{theorem}

\begin{proof}  
It is easy to see that $wf$ satisfies Assumption~\ref{ass:f} with a growth rate $M_{\omega f}=M_\omega+M_f$ by Lemma~\ref{lem:mis}.  Note that 
$$
        \begin{aligned}
        |\pi_{N}(f)-\pi(f)| &= \left|\frac{I_N(f \omega)}{I_N(\omega)}-\frac{q(\omega f)}{q(\omega)}\right|\\
        &\leq \frac{1}{|I_N(\omega)|}\left|I_N(f\omega)-q(f \omega)\right|+\left|\frac{q(f \omega)}{I_N(\omega)q(\omega)}\right|\left|I_N(\omega)-q(\omega)\right|.
        \end{aligned}
$$
If \(\max\{M_{\omega},M_{\omega}+M_f\}<0\), then by Lemma~\ref{lem:negative_bvhk}, $g_1(\bm u) = \omega\circ\mathcal{T}(\bm u)$ and $g_2(\bm u) = \omega f\circ\mathcal{T}(\bm u)$ are both of BVHK. On the probability-one event in Assumption~\ref{ass:rqmc}, the
Koksma--Hlawka inequality \eqref{eq:KH_inq} gives
\[
    |I_N(\omega)-q(\omega)|
    =
    \left|
    \frac1N\sum_{i=1}^N g_1(\bm{u}_i)
    -
    \int_{(0,1)^d}g_1(\bm{u})\,d\bm{u}
    \right|
    \le
    B_d V_{\mathrm{HK}}(g_1)\frac{(\log N)^{d-1}}{N}.
\]
Thus, there exists $N_0>0$ such that for any $N\ge N_0$, $|I_N(\omega)-q(\omega)|< 0.5q(\omega)$. Therefore, for every $N\ge N_0$, $\mathbb{P}_{\mathrm{RQMC}}(I_N(\omega)>0.5q(\omega))=1$. As a result, with a probability one, 
\begin{equation*}
|\pi_{N}(f)-\pi(f)|\le \frac{2}{q(\omega)}\left|I_N(f\omega)-q(f \omega)\right|+\frac{2|q(f \omega)|}{q(\omega)^2}\left|I_N(\omega)-q(\omega)\right|.
\end{equation*}
Taking the $p$-th power, applying $(a+b)^p \le 2^{p-1}(a^p+b^p)$ and $|q(f\omega)|=|\int_{\mathbb{R}^d}f(\bm x)\bar \pi(\bm x)d\bm x|<\infty$, and taking expectations on both sides yields
    \begin{equation}\label{eq:e0}
        \begin{aligned}
    \mathbb{E}\left[\left|\pi_{N}(f)-\pi(f)\right|^p \right]
    &\leq 2^{p-1} \left(\frac{2}{q(\omega)}\right)^p \mathbb{E}\left[\left|I_N(f \omega)-q(f \omega)\right|^p\right] \\
    &\quad + 2^{p-1} \left(\frac{2|q(f\omega)|}{q(\omega)^2}\right)^p \mathbb{E}\left[\left|I_N(\omega)-q(\omega)\right|^p\right].
        \end{aligned}
    \end{equation}
Applying the first statement in Theorem~\ref{thm:rqmc} with  $g=\omega f$ and $g=\omega$ gives the first claim in the theorem.

We now turn to  the second claim \(\mathrm{(ii)}\)  in which $g_1$ and $g_2$ are no longer of BVHK. To bound the $L_p$-error, we follow the proof strategy of \red{\cite[Theorem 2.2]{deligiannidis2024}}. Define the event
    \[
    \mathcal{E} := \left\{\left|I_N(\omega)-q(\omega)\right|>0.5q(\omega)\right\}.
    \]
    By the law of total probability, the error can be decomposed as
    \begin{equation}\label{eq:decom}
    \mathbb{E}\left[\left|\pi_{N}(f)-\pi(f)\right|^p\right] = \mathbb{E}\left[\left|\pi_{N}(f)-\pi(f)\right|^p \mathds{1}_{\mathcal{E}}\right] + \mathbb{E}\left[\left|\pi_{N}(f)-\pi(f)\right|^p \mathds{1}_{\mathcal{E}^{c}}\right].
    \end{equation}
    For the first term, using $ |\pi_{N}(f)| = \left|\sum_{i=1}^{N}\frac{\omega\left(\bm{\bm{x}}_i\right) }{\sum_{j=1}^{N}\omega\left(\bm{\bm{x}}_j\right)} f\left(\bm{\bm{x}}_i\right)\right| \le \max_{1 \leq i \leq N} |f\left(\x_i\right)|$ and $(a-b)^p \le 2^{p-1}(|a|^p+|b|^p)$, we have
    \begin{equation*}
                \begin{aligned}
            \mathbb{E}\left[\left|\pi_{N}(f)-\pi(f)\right|^p \mathds{1}_{\mathcal{E}}\right] \leq&~ \mathbb{E}\left[\left(\max _{1 \leq i \leq N} |f\left(\x_i\right)|+|\pi(f)|\right)^p \mathds{1}_{\mathcal{E}}\right] \\
            \leq&~ 2^{p-1} \mathbb{E}\left[\left(\max _{1 \leq i \leq N} |f\left(\x_i\right)|\right)^p \mathds{1}_{\mathcal{E}}\right] +2^{p-1}|\pi(f)|^p \mathbb{P}_{\mathrm{RQMC}}(\mathcal{E}).
            \end{aligned}
    \end{equation*}
    Since \(M_{\omega}>0\) can be taken arbitrarily small, using Markov's inequality with $1 \le s < \alpha/M_{\omega}$ and then applying Theorem~\ref{thm:rqmc} with $g=\omega$, for any $\epsilon\in (0,1)$, there exists a constant $C_{\epsilon,s}> 0$, independent of \(N\), such that 
    \begin{align*}
        \mathbb{P}_{\mathrm{RQMC}}(\mathcal{E}) &= \mathbb{P}_{\mathrm{RQMC}}\left(\left|I_N(\omega)-q(\omega)\right|>0.5q(\omega)\right) \leq \frac{\mathbb{E}\left[\left|I_N(\omega)-q(\omega)\right|^s\right]}{(0.5q(\omega))^s} \\
        &\le \frac{C_{\epsilon,s}}{(0.5q(\omega))^s} N^{-(1-\epsilon)s}= C_{\mathcal{E}} N^{-(1-\epsilon)s},
    \end{align*}
    where $C_{\mathcal{E}}=C_{\epsilon,s}(0.5q(\omega))^{-s}$.
    % Defining $C_{\mathcal{E}} = (C_{\omega, s}/c)^s$, we have $\mathbb{P}(\mathcal{E}) \le C_{\mathcal{E}} N^{-(1-\epsilon)s}$. 
    Then, applying H\"{o}lder's inequality with exponents $v/p$ and $(1-p/v)^{-1}$, we obtain
    $$
        \begin{aligned}
        \mathbb{E}\left[\left(\max\limits_{1 \leq i \leq N} |f\left(\x_i\right)|\right)^p \mathds{1}_{\mathcal{E}}\right]
        \leq& ~\mathbb{E}\left[\left(\max\limits_{1 \leq i \leq N} |f\left(\x_i\right)|\right)^v\right]^{p/v} \times \left[\mathbb{P}_{\mathrm{RQMC}}(\mathcal{E})\right]^{1-p/v} \\
        \leq & \left( N q(|f|^v) \right)^{p/v} \left( C_{\mathcal{E}} N^{-(1-\epsilon) s} \right)^{1-p/ v} \\
        \leq &~ C_{\mathcal{E}}^{1-p/ v}q\left(|f|^v\right)^{p/v} N^{p/v-(1-\epsilon) s(1-p/ v)},
        \end{aligned}
        $$
        where we use the fact that     $$\mathbb{E}\left[\left(\max\limits_{1 \leq i \leq N}| f\left(\x_i\right)|\right)^v\right] \le \mathbb{E}\left[\sum_{i=1}^N \left|f\left(\x_i\right)\right|^v \right] = Nq(|f|^v).$$
        We now take a large enough $s$ to ensure
        \begin{equation*}% E2-2: removed number (3.15); original label: eq:ps
            p/v-(1-\epsilon)s(1-p/ v) \le -(1-\epsilon)p.
        \end{equation*} 
        We thus have
    \begin{equation}\label{eq:e1}
        \mathbb{E}\left[\left|\pi_{N}(f)-\pi(f)\right|^p \mathds{1}_{\mathcal{E}}\right] \le C_1 N^{-(1-\epsilon) p}
    \end{equation}
for $C_1 =2^{p-1}(C_{\mathcal{E}}^{1-p/ v}q\left(|f|^v\right)^{p/v}+|\pi(f)|^pC_{\mathcal{E}}) > 0$.

    On the event $\mathcal{E}^c$, $I_N(\omega) > 0.5 q(\omega) > 0$. It follows that 
    \[
    |\pi_{N}(f)-\pi(f)|\mathds{1}_{\mathcal{E}^{c}} \le \frac{2}{q(\omega)}\left|I_N(f\omega)-q(f \omega)\right| + \frac{2|q(f\omega)|}{q(\omega)^2}\left|I_N(\omega)-q(\omega)\right|.
    \]
Then $\mathbb{E}[|\pi_{N}(f)-\pi(f)|^p\mathds{1}_{\mathcal{E}^{c}}]$ has an upper bound as in the right hand side of \eqref{eq:e0}.
Applying Theorem~\ref{thm:rqmc} with  $g=\omega f$ and $g=\omega$, and noticing again that $M_\omega>0$ is arbitrarily small, it follows that
\begin{equation}\label{eq:e2}
    \mathbb{E}[|\pi_{N}(f)-\pi(f)|^p\mathds{1}_{\mathcal{E}^{c}}] \le C_2 N^{-(\beta-\epsilon) p}+C_3 N^{-(1-\epsilon) p} \le (C_2+C_3)N^{-(\beta-\epsilon) p}
\end{equation}
for constants $C_2 > 0$, $C_3 > 0$, and $\beta = 1- p M_{f}/\alpha \le 1$.
By \eqref{eq:decom} and \eqref{eq:e1} together with \eqref{eq:e2}, the result follows immediately.
        \end{proof}

        \begin{remark}\label{rem:liu_snis_growth}
% \red{
% 大概说明
% The assumption in
% Theorem~\ref{thm:SNIS}$\mathrm{(ii)}$ that $M_\omega>0$ can be taken arbitrarily
% small provides sufficiently high-order RQMC error bounds
% for the normalizing denominator $I_N(\omega)$, which
% are used to control the contribution of the exceptional
% event in the SNIS error decomposition. Indeed, for a fixed $M_\omega>0$, applying
% Theorem~\ref{thm:rqmc} to $\omega$ with
% $1\le s<\alpha/M_\omega$ and using
% Markov's inequality gives
% \(
% \mathbb P_{\mathrm{RQMC}}(\mathcal E)
% =
% \mathcal O\left(
% N^{-s(1-sM_\omega/\alpha)+\epsilon}
% \right),
% \)
% where $\mathcal E$ is the exceptional
% event defined in the proof of
% Theorem~\ref{thm:SNIS}. Consequently, the corresponding
% H\"older bound contains a term of order
% \(
% \mathcal O\left(
% N^{p/v-
% s(1-sM_\omega/\alpha)(1-p/v)
% +\epsilon}
% \right).
% \)
% For a fixed positive $M_\omega$, the
% admissible moment order $s$ is restricted
% by $sM_\omega<\alpha$, and additional
% compatibility conditions on $p$, $v$,
% and $s$ are needed to obtain a positive
% convergence rate. However, when $M_\omega>0$ can be taken arbitrarily
% small, the exceptional event contribution can be
% controlled at the required rate, while the
% contribution of $M_\omega$ to the numerator
% growth rate $M_{\omega f}=M_\omega+M_f$ can
% be absorbed into the arbitrarily small loss
% in the convergence exponent.
% Consequently, the resulting exponent
% is determined by $M_f$, as stated in
% Theorem~\ref{thm:SNIS}.}

An unbounded function that satisfies Assumption~\ref{ass:f} 
with an arbitrarily small growth rate $M>0$ is referred to as
 a ``QMC-friendly" function \cite{he2024}. If both functions
  $\omega$ and $f$ are QMC-friendly, then applying Theorem \ref{thm:SNIS} yields 
$
\set{\E\left|\pi_{N}(f)-\pi(f)\right|^{p}}^{1/p} = \mathcal{O}\left(N^{-1+\epsilon}\right)
$
for any $\epsilon>0$ and $1 \leq p < v$.
\red{Compared with the IID-based $L_p$-error rate $\mathcal{O}(N^{-1/2})$ of
 \cite[Theorem~2.2]{deligiannidis2024}, which is established under moment conditions on
  the importance weights, the test function, and their product, the preceding
   result yields a faster rate for admissible values of $p$.
    However, the IID-based analysis does not impose the smoothness,
     derivative-growth, or sub-Gaussian tail conditions required here.
      Thus, the faster RQMC rate is conditional on the additional regularity
       and growth assumptions stated above.}
\end{remark}

\section{Admissible transport maps as the proposals}\label{subsec:apply}
Section~\ref{sec:QMC} focuses on proposals with independent marginals, which is limited in practice. For example, if the proposal is a general Gaussian distribution, Theorem \ref{thm:SNIS}
cannot be applied directly. Note that Gaussian distributions can be generated via a standard Gaussian through an affine transformation. The standard Gaussian can thus serve as a base distribution, and then a general Gaussian proposal is a pushforward measure of the base distribution. Generally, suppose that $\bm Y$ has density $q_0$, where $q_0$ is a (simple) base distribution density. For a diffeomorphic mapping $\tau$, we take 
\begin{equation*}% E2-2: removed number (4.1); original label: eq:NFdist
    \bm X=\tau(\bm Y)\sim q(\bm x)=q_0(\tau^{-1}(\bm x))|\det J_\tau(\tau^{-1}(\bm x))|^{-1}
\end{equation*}
as the pushforward measure of $q_0$ by $\tau$, where $J_\tau(\cdot)$ denotes the Jacobian matrix of $\tau$. The mapping $\tau:\mathbb R^d\to\mathbb R^d$ is called the transport map. This section extends the results in Section~\ref{sec:QMC}  to more flexible proposals via specific transport maps. Due to the simplicity of the base distribution, we follow the setting  in Section~\ref{sec:QMC} to assume that  $q_0(\bm x)$ has independent marginals, which can be generated via the inversion method as in \eqref{eq:tau1}. Despite this, the resulting proposal $q(\bm x)$ can be  expressive for cleverly designed transport maps $\tau$, such as normalizing flows \cite{liu2024transport}. The problem can then be written as
$$
        \pi(f) = \frac{\mathbb{E}_{q}[\omega(\bm X)f(\bm X)]}{\mathbb{E}_{q}[\omega(\bm X)]} = \frac{\mathbb{E}_{q_0}[\omega\circ\tau (\bm Y)f\circ\tau (\bm Y)]}{\mathbb{E}_{q_0}[\omega\circ\tau(\bm Y)]}= \frac{q_0(\omega_\tau f_\tau)}{q_0(\omega_\tau)},
$$
where \(\omega=\bar\pi/q\), \(f_\tau(\bm y)=f(\tau(\bm y))\), and
\[
\omega_\tau(\bm y)=\omega(\tau(\bm y))
=\frac{\bar\pi(\tau(\bm y))}{q(\tau(\bm y))}
=\frac{\bar\pi(\tau(\bm y))|\det J_\tau(\bm y)|}{q_0(\bm y)}.
\]
This formulation reduces the analysis to the setting of Section \ref{sec:QMC} with the base proposal \(q_0\), applied to the transformed integrands \(\omega_\tau f_\tau\) and \(\omega_\tau\). Therefore, the \(L_p\)-error analysis in Theorem \ref{thm:SNIS} is applicable provided that \(q_0\) satisfies the required tail condition and the transformed integrands satisfy the growth and moment conditions required there. In the negative-growth case, no tail condition on \(q_0\) is needed; in the nonnegative-growth case, the sub-Gaussian-type tail condition is imposed on \(Y\sim q_0\). We next provide some sufficient conditions to examine the growth rates of \(\omega_\tau f_\tau\) and \(\omega_\tau\).

\begin{theorem}\label{thm:tau}
\red{Assume that $\tau\in C^d(\mathbb R^d;\mathbb R^d)$, and} for any $\bm{a} \in \mathbb{N}_0^d$ with $1\le |\bm{a}| \leq d$, there exist constants \(B_\tau,M_\tau\ge0\) such that
\begin{equation}\label{eq:dtau}
    |D^{\bm{a}}\tau_i(\x)| \le B_\tau \exp(M_\tau\|\x\|^2), \quad \forall  \x\in\mathbb R^d
\end{equation}
for \(i=1,\ldots,d\), where $\tau_i$ denotes the $i$-th component of $\tau$.
Suppose that  the function \(g(\x)\) defined over $\mathbb{R}^d$  
satisfies Assumption~\ref{ass:f} with a growth rate \(M_g\in \mathbb{R}\) 
and a constant $C_g>0$. Assume that there exist constants 
$C_\tau>0,C_{\tau}^{\prime}\in \mathbb{R}$ such that 
$\|\tau(\x)\|^2 \le C_\tau\|\x\|^2+C_{\tau}^{\prime}$ 
if $M_g\ge 0$ and otherwise $\|\tau(\x)\|^2 \ge C_\tau\|\x\|^2+C_{\tau}^{\prime}$. 
Then $g_\tau=g\circ\tau$ satisfies  Assumption~\ref{ass:f} with a growth rate \(M_{g_\tau}=M_g C_\tau+dM_\tau\). 
\end{theorem}
\begin{proof}
\red{Since $g$ and $\tau$ are $C^d$, their composition $g_\tau$ is also $C^d$ 
 and the following Faa di Bruno expansion is valid.}
 We only prove the case of $M_g\ge 0$ and the case of $M_g<0$ can be proved similarly.
 By Assumption \ref{ass:f} for $g$ and the upper bound on $\tau$, we have
\[
 |g(\tau(x))|
 \le C_g \exp\{M_g\|\tau(x)\|^2\}
 \le C_g\exp\{M_g C_\tau^{\prime}\}
        \exp\{(M_g C_\tau + dM_\tau)\|x\|^2\}.
\]
Let $\tau_{i}$ be the $i$-th component of $\tau$. By the Faa di Bruno formula \cite{cons1996}, for any $\x \in \mathbb{R}^d$ and $\bm{a} \in \mathbb{N}_0^d$ with $1\le |\bm{a}|\le d$, $D^{\bm{a}} g_\tau(\x)$ is a sum of finite terms of the form 
            \begin{equation*}% E2-2: removed number (4.3); original label: eq:faa
                 D^{\boldsymbol{\lambda}} g(\tau(\x)) \prod_{j=1}^{s} \prod_{i=1}^d (D^{\bm{\ell}_j} \tau_{i}(\x))^{\bm{k}_{j,i}},
            \end{equation*}
            where $\boldsymbol{\lambda} \in \mathbb{N}_0^d, 1 \leq|\boldsymbol{\lambda}| \leq |\bm{a}|$, $1 \le s \le |\bm{a}|$, $\bm{k}_{j}\in \mathbb{N}_0^d\setminus \{{\bm{0}}\}$ satisfying $\sum_{j=1}^{s} \bm{k}_{j} = \boldsymbol{\lambda}$, and $\bm{\ell}_{j}\in \mathbb{N}_0^d\setminus \{{\bm{0}}\}$ such that $\sum_{j=1}^{{s}} |\bm{k}_{j}| \bm{\ell}_{j}=\boldsymbol{a}$.
Since all $\tau_i(\bm x)$ satisfy \eqref{eq:dtau}, it follows that
    \[
    \begin{aligned}
    \prod_{j=1}^{s} \prod_{i=1}^d \left|D^{\bm{\ell}_j} \tau_{i}(\x)\right|^{\bm{k}_{j,i}} 
    &\le \prod_{j=1}^{s} \prod_{i=1}^d \left( B_\tau \exp(M_\tau\|\x\|^2) \right)^{\bm{k}_{j,i}} \\
    &= B_\tau^{|\boldsymbol{\lambda}|} \exp(M_\tau|\boldsymbol{\lambda}|\|\x\|^2) \le \max(1, B_\tau^d) \exp(dM_\tau\|\x\|^2),
    \end{aligned}
    \]
where $B_\tau>0$ is a constant. 
Since $g$ satisfies Assumption~\ref{ass:f} with a growth rate \(M_g\ge 0\)
 and a constant $C_g>0$, and $\|\tau(\x)\|^2 \le C_\tau\|\x\|^2+C_{\tau}^{\prime}$,
  we have $$|D^{\boldsymbol{\lambda}} g(\tau(\x))| \le C_g \exp(M_g\|\tau(\x)\|^2)\red{\le}C_g\exp(M_gC_\tau\|\x\|^2+M_gC_\tau^{\prime}).$$
Hence, we obtain
    \[
    \left|D^{\boldsymbol{\lambda}} g(\tau(\x)) \prod_{j=1}^{s} \prod_{i=1}^d (D^{\bm{\ell}_j} \tau_{i}(\x))^{\bm{k}_{j,i}}\right| \le C \exp\left((M_gC_{\tau}+dM_\tau)\|\x\|^2\right),
    \]
    where $C = C_g \max(1, B_\tau^d) \exp(M_gC_{\tau}^{\prime})$. The number of terms in the Faa di Bruno expansion depends only on \(d\). Therefore, after summing all terms, \(g_\tau\) satisfies Assumption~\ref{ass:f} with the growth rate $ M_{g_\tau}$, completing the proof. 
        \end{proof}

Now, we consider a specific form of the transformation  $\tau(\x) = T(L\x+\bm{\mu})$, where $T(\x) = \left(T_1(x_1),T_2(x_2),\ldots,T_d(x_d)\right)$ is a component-wise transformation with each $T_i$ being a strictly monotone function, $L \in \mathbb{R}^{d\times d}$ is a nonsingular lower-triangular matrix (scaled transformation), and $\bm{\mu} \in \mathbb{R}^d$ (shifted transformation). Let $\lambda_{\max}(\cdot)$ and $\lambda_{\min}(\cdot)$ be the largest and the smallest eigenvalues of a matrix, respectively.

 \begin{proposition}\label{pro:T}
\red{Suppose that $T_i\in C^d(\mathbb R)$ for all $i=1,\ldots,d$.}
If there exist constants $M_T,B_T\ge0$ such that for all $1\le m \le d$ and $1\le i\le d$,
\begin{align*}
|D^m T_i(y)| \le B_{T}\exp(M_T|y|^2), \quad \forall y\in\mathbb{R},
\end{align*}
then the transport map $\tau(\bm x)=T(L\x+\bm{\mu})$ satisfies \eqref{eq:dtau} with growth rate \(M_\tau=(1+\varepsilon)M_T\lambda_{\max}(L^{\mathrm{T}}L)\) for any arbitrarily small $\varepsilon>0$.
If $|T_i(y)|^2 \le C_T|y|^2+C_T^{\prime}$ for some constants $C_T>0$ and $C_T^{\prime}\in\mathbb{R}$, then $\|\tau(\x)\|^2 \le C_\tau\|\x\|^2+C_{\tau}^{\prime}$ for $C_\tau =(1+\varepsilon)C_T\lambda_{\max}(L^{\mathrm{T}}L)$ and a constant $C_{\tau}^{\prime}$.
If $|T_i(y)|^2 \ge c_T|y|^2+c_T^{\prime}$ for some constants $c_T>0$ and $c_T^{\prime}\in\mathbb{R}$, then $\|\tau(\x)\|^2 \ge c_\tau\|\x\|^2+c_{\tau}^{\prime}$ for $c_\tau = (1-\varepsilon)c_T\lambda_{\min}(L^{\mathrm T}L)$ and a constant $c_{\tau}^{\prime}$.
\end{proposition}
        
\begin{proof}
            \red{Since $T_i\in C^d$ and the inner map is affine, $\tau\in C^d(\mathbb R^d;\mathbb R^d)$.} Let $L_i$ be the $i$-th row of $L$. Then, $\tau_i(\x) = T_i(L_i\x+\mu_i)$. By the Faa di Bruno formula \cite{cons1996}, for any $\x \in \mathbb{R}^d$ and $\bm{a} \in \mathbb{N}_0^d$ satisfying $1\le |\bm{a}|\le d$, we have
            \begin{shrinkeq}{-1ex}
                \begin{align*}
            D^{\bm{a}}\tau_i(\x)  = D^{|\bm{a}|}T_i(y_i)\prod_{j=1}^d (L_{i,j})^{a_j},
            \end{align*}
            \end{shrinkeq}
            where $ y_i = L_i\x+\mu_i$. Since $|D^mT_i(y)| \le B_T\exp(M_T|y|^2)$ for any $1\le m\le d$, by Young's inequality, it follows that for any $\varepsilon > 0$,
            $$
            \begin{aligned}
                &|D^{\bm{a}}\tau_i(\x)|\le \max_{1\le i\le d} \left| \prod_{j=1}^d (L_{i,j})^{a_j} \right| B_T\exp(M_T|L_i\x+\mu_i|^2)\\ &\le \max_{\substack{1\le i\le d\\1\le |\bm{a}|\le d}} \left| \prod_{j=1}^d (L_{i,j})^{a_j} \right| B_T\exp\left(M_T\|\mu\|_{\infty}^2(1+1/\varepsilon)\right)\exp(M_T\lambda_{\max}(L^{\mathrm{T}}L)(1+\varepsilon)\|\x\|^2),
            \end{aligned}
            $$
            where  we use the fact that $M_T \ge 0$ and
            $$
            (L_i\x)^2 = \x^{\T}L_i^{\T}L_i\x \le \lambda_{\max}(L_i^{\T}L_i)\|\x\|^2 \le \lambda_{\max}(L^{\mathrm{T}}L)\|\x\|^2.
            $$ 
            In addition, if $|T_i(y)|^2 \le C_T|y|^2+C_T^{\prime}$,
            \[ 
            \begin{aligned} 
            \|\tau(\x)\|^2 = \sum_{i=1}^d |T_i(y_i)|^2 \le C_T\sum_{i=1}^d |y_i|^2+dC_T'= C_T\|L\x+\bm\mu\|^2+dC_T'. 
            \end{aligned} \] 
            Using Young's inequality, $\|L\x+\bm\mu\|^2 \le (1+\varepsilon)\|L\x\|^2 + \left(1+1/\varepsilon\right)\|\bm\mu\|^2$,
           we obtain 
            \[ \|\tau(\x)\|^2 \le (1+\varepsilon)C_T\lambda_{\max}(L^{\mathrm{T}}L)\|\x\|^2 + C_T\left(1+1/\varepsilon\right)\|\bm\mu\|^2 + dC_T'. \]
            Similarly, we can prove the claim for the case $|T_i(y)|^2 \ge c_T|y|^2+c_T^{\prime}$.
        \end{proof}

        \begin{remark}\label{rem:linear}
            If $T(\x) = \x$, i.e., $\tau(\x) = L\x+\bm{\mu}$, then \(C_T=c_T = 1\), \(C_T'=c_T' = 0\), and \(M_T=0\). If \(f\) satisfies Assumption~\ref{ass:f} with \(M_f\in \mathbb{R}\), by Theorem~\ref{thm:tau} and Proposition~\ref{pro:T}, $f_\tau=f\circ\tau$ has a growth rate $M_{f_\tau}=M_fC_\tau$ with 
            \begin{equation*}% E2-2: removed number (4.4); original label: eq:linear
                 C_\tau=
            \begin{cases}
                (1+\varepsilon)\lambda_{\max}(L^{\mathrm{T}}L),&\quad M_f\ge 0\\
                (1-\varepsilon)\lambda_{\min}(L^{\mathrm{T}}L),&\quad M_f< 0
            \end{cases}
            \end{equation*}
            for arbitrarily small $\varepsilon>0$.
            The growth rate $M_{\omega_\tau}$ for  $\omega_\tau=\omega\circ\tau$ can be obtained similarly. Clearly, a negative-growth rate $M_\omega<0$ for $\omega$ implies a negative-growth rate  $M_{\omega_\tau}<0$ for $\omega_\tau$. Also, if $M_\omega>0$ is arbitrarily small, we end up with an arbitrarily small
            $M_{\omega_\tau}>0$ for $\omega_\tau$. Thus location--scale transformations preserve the regimes of $\omega$ in Theorem~\ref{thm:SNIS}. RQMC-SNIS achieves an $L_p$-error rate of nearly $\mathcal{O}(N^{-1+p M_f\lambda_{\max}(L^{\mathrm{T}}L)/\alpha})$ if $M_f\ge 0$ and $M_\omega>0$ is arbitrarily small.
        \end{remark}

        In modern Bayesian computation, the target distribution $\pi$ may be multimodal or skewed or spiky. 
        It usually calls for an expressive proposal $q$ to fit such a complicated target distribution. We next consider the composed maps $\tau = \tau^{K}\circ\cdots\circ\tau^{2}\circ\tau^{1}$ as the transport map studied by Liu~\cite{liu2024transport}, where $K\ge 1$ is the number of layers, $\tau^{k}(\x) = T^{k}(L^{k}\x+\bm{\mu}^{k})$, $T^{k}$ is a component-wise strictly monotone transformation with $T^{k}(\x) = (T_1^{k}(x_1),T_2^{k}(x_2),\ldots,T_d^{k}(x_d))$, $L^{k} \in \mathbb{R}^{d\times d}$ is a lower-triangular matrix and $\bm{\mu}^{k} \in \mathbb{R}^d$, $k =1,2,\ldots,K$. 
         % Denote $M_L = \max\limits_{1\le k\le K}\{\lambda_{\max}((L^k)^{\mathrm{T}}L^k)\}$.
         
          \begin{proposition}\label{pro:tauk}
              Suppose that there exist common constants \(c_T,C_T,M_T,B_T\) such that all component maps \(T_i^k\), \(1\le i\le d\), \(1\le k\le K\), satisfy the assumptions of Proposition~\ref{pro:T} with these constants and $\tau = \tau^{K}\circ\cdots\circ\tau^{2}\circ\tau^{1}$ for $K\ge 1$ with $\tau^{k}(\x) = T^{k}(L^{k}\x+\bm{\mu}^{k})$. Define 
              $$
              M_L := \max\limits_{1\le k\le K} \lambda_{\max}\bigl((L^k)^{\mathrm T}L^k\bigr), \quad m_L := \min\limits_{1\le k\le K} \lambda_{\min}\bigl((L^k)^{\mathrm T}L^k\bigr)>0.
              $$
              Then the conclusions in Proposition~\ref{pro:T} hold with 
            $$c_\tau = c_*^{(K)} := \left((1-\varepsilon)c_Tm_L\right)^K, \quad C_\tau = C_*^{(K)} = \left((1+\varepsilon)C_TM_L\right)^K,
              $$
              and the  growth rate $M_\tau = M_\tau^{(K)}$ where \(M_\tau^{(0)}=0\) and, for \(k=1,\ldots,K\),
              $$
              M_\tau^{(k)} = (1+\varepsilon)M_TM_LC_{*}^{(k-1)}+dM_\tau^{(k-1)}. 
              $$ 
              % with $C_{*}^{(0)} = 1,$, $C_{*}^{(k)} = \left(C_T(1+\rho)M_L\right)^k$.
          \end{proposition}
          \begin{proof}
              \red{Each layer is $C^d$ by Proposition~\ref{pro:T}, hence so is their finite composition.} Similar to the proof of Theorem 4.5 of \cite{liu2024transport}, by induction on $K$, we establish the result.
          \end{proof}

As suggested by Proposition~\ref{pro:tauk}, to gain a small growth rate $M_\tau$ for the transport map $\tau$, one may choose $T^{k}_i$ with bounded derivatives to achieve $M_T=0$. By doing so, $M_\tau=0$ and the composite function $f_\tau = f\circ\tau$ has the growth rate $M_{f_\tau}=M_fC_\tau=M_f\left((1+\varepsilon)C_TM_L\right)^K$ for $M_f>0$. If $M_f>0$ cannot be arbitrarily small, $f_\tau$ has a severe growth rate, especially for a large number of layers $K$. 
Constructing an optimal transport map \(\tau\) lies beyond the scope of this work; interested readers are referred to Liu \cite{liu2024transport} for relevant discussions.

    \section{Numerical experiments}\label{sec:experiments}
In this section, we investigate a toy Bayesian inverse problem 
% and a Bayesian logistic regression problem 
to examine the effect of the parameters in Theorem \ref{thm:SNIS} for RQMC-SNIS, driven by classical scrambled Sobol' point sets with direction numbers from \cite{joe2003}, which satisfy Assumption \ref{ass:rqmc}. An analysis involving the Bayesian logistic regression model is further explored in Appendix~\ref{app:logistic}. 
We perform $R$ independent randomizations of the underlying QMC point set to obtain $R$ replicated estimator values, denoted by $\pi_N^{(1)}(f),\ldots,\pi_N^{(R)}(f)$.
% We repeat the estimator independently $R$ times 
% to estimate the $L_p$-error rate of the estimators.
The empirical $L_p$-error is computed as
\begin{equation}\label{eq:var}
	\widehat{L_p}= \left[\frac{1}{R} \sum_{i=1}^R \left|\pi^{(i)}_N(f) - \hat\pi(f)\right|^p\right]^{1/p},
\end{equation}
where $\hat\pi(f) = \frac{1}{R} \sum_{i=1}^R\pi^{(i)}_N(f)$.
% denotes the final estimate of the posterior expectation $\pi(f)$.
In the following experiments, we fix $R = 50$, $f(\x) = \|\x\|^2$ which satisfies Assumption \ref{ass:f} with arbitrarily small $M_f>0$, and we consider a linear transport map $\tau(\bm{x}) = \bm{\mu}+L\bm x$ with $\bm{\mu} \in \mathbb{R}^d$ and $L \in \mathbb{R}^{d \times d}$ satisfying $L^{\mathrm{T}}L = \bm{\Sigma}$ for a positive definite matrix $\bm{\Sigma}$.

Consider a Bayesian inverse problem with the forward model $\bm{y} = \mathcal{G}(\z)+\bm \varsigma$, where  $\bm{y}\in \mathbb{R}^d$ is the observed data,  $\bm{z} \in \mathbb{R}^d$ is the unknown parameter,  $\mathcal{G}$ is the response operator, and $\bm \varsigma\in \mathbb{R}^d$ is the noise vector. Assume that $\bm \varsigma \sim \mathcal{N}(\bm 0,(1/n)I_d)$, where $n$ is the level of noise. The likelihood $l(\z)$ is proportional to $\exp \left\{-n \Psi(\z)\right\}$, where $\Psi(\z) = \frac{1}{2}\|\bm{y}-\mathcal{G}(\z)\|^2$. Let $\pi_0(\z)$ be the prior distribution. By Bayes' formula, the posterior distribution is given by
$\pi(\z) \propto \bar\pi(\z):=\exp (-n \Psi(\z)) \pi_0(\z)$
with an intractable constant.

Our aim is to estimate  the posterior expectation of interest given by
$$
\pi(f):=\mathbb{E}_\pi[f(\bm{Z})]=\frac{\int_{\mathbb{R}^d} f(\z) \exp (-n \Psi(\z)) \pi_0(\z) d\z}{\int_{\mathbb{R}^d} \exp (-n \Psi(\z)) \pi_0(\z) d\z}.
$$
Following \cite{he2024}, we consider a Gaussian prior $\pi_0=N\left(\bm{\mu_0}, \bm{\Sigma}_0\right)$ and a linear mapping with a small nonlinear perturbation as the response operator to analyze the effect of the model parameters. Specifically, we take $\mathcal{G}(\z)=\z+ \lambda \mathcal{F}(\z)$ with $\lambda>0$ and $\mathcal{F}(\z)=\left(z_1 e^{-z_1^2}, \ldots, z_d e^{-z_d^2}\right)^{\mathrm{T}}$. For simplicity, we take $\bm{y}=\bm{0}$, yielding
$$
\Psi(\z)=\frac{1}{2}\|\mathcal{G}(\z)\|^2=\frac{1}{2} \sum_{i=1}^d z_i^2\left(1+\lambda e^{-z_i^2}\right)^2.
$$
Note that $\Psi(\z) \geq \frac{1}{2}\|\z\|^2$. By Example 4.9 of \cite{he2024}, $\Psi(\z)$ satisfies Assumption \ref{ass:f} with arbitrarily small $M_\Psi>0$. Hence, the likelihood $l(\z)$ satisfies Assumption \ref{ass:f} with $M_l = -\frac{n}{2}+M_\Psi<0$. For the Gaussian prior $N\left(\bm{\mu_0}, \bm{\Sigma}_0\right)$, we set $\bm{\mu_0}=\bm{1}$ and $\bm{\Sigma}_0 = I_d$. Therefore, the unnormalized posterior density $\bar\pi(\z)$ has a growth rate $M_{\bar\pi} =  -(n+1)/2+\varepsilon_0<0$, where $\varepsilon_0>0$ is arbitrarily small.

Note that $\bm{\mu}^*=\mathbf{0}$ is the global minimizer of $\Psi(\z)$, and $\bm{\Sigma}^*=\left(\nabla^2 \Psi\left(\bm{\mu}^*\right)\right)^{-1}=\kappa I_d$ with $\kappa=(1+\lambda)^{-2}$. In our experiments, we vary the model parameter $\kappa$, the dimension $d$, and the moment parameter $p$ to show the performance of various RQMC-SNIS estimators. We consider five proposals as follows: 
\begin{itemize}
    \item PriorIS: Gaussian proposal, $\bm{\mu}=\bm{\mu_0}=\mathbf{1}$ and $\bm{\Sigma} = \bm{\Sigma}_0 = I_d$;
    \item ODIS: Gaussian proposal,  $\bm{\mu}=\bm{\mu}^*=\mathbf{0}$ and $\bm{\Sigma} = \bm{\Sigma}_0 = I_d$;
    \item LapIS: Gaussian proposal, $\bm{\mu}=\bm{\mu}^*=\mathbf{0}$ and $\bm{\Sigma}=\bm{\Sigma}^*/n=(\kappa/n) I_d$;
    \item tIS0: Linear t proposal, $\bm{\mu}=\mathbf{1}$, $\bm{\Sigma}= I_d$ and the degrees of freedom $\nu = 5$;
    \item tIS: Linear t proposal, $\bm{\mu}=\mathbf{0}$, $\bm{\Sigma}= (\kappa/n) I_d$ and the degrees of freedom $\nu = 5$.
\end{itemize}
For simplicity, throughout the experiments, we fix $n = 20$.
% which leads a bad performance for LapIS in Figure 4 of \cite{he2024} with $\kappa = 1/4$. 
Regarding the impact of different noise levels \(n\), we refer the reader to \cite{he2024}.

Since $M_f > 0$ is arbitrary small, the function $f_\tau=f\circ\tau$ has an arbitrary small growth rate $M_{f_\tau}>0$; see Remark~\ref{rem:linear} for a discussion. For a proposal density $q$, we write $M_{1/q}$ for the growth rate of $1/q$. Applying Lemma~\ref{lem:mis} with Gaussian proposals, we then have $$M_{\omega} = M_{1/q} +M_{\bar\pi}=\frac 12 \lambda_{\max}\bigl(\bm{\Sigma}^{-1}\bigr)+\varepsilon_1-\frac{n+1}2+\varepsilon_0 = \frac 12 \lambda_{\max}\bigl(\bm{\Sigma}^{-1}\bigr)-\frac{n+1}2+\varepsilon_2.$$
where $\varepsilon_1>0$ is arbitrarily small and $\varepsilon_2 = \varepsilon_0+\varepsilon_1$.
For LapIS, the growth rate of $\omega$ is $$M_{\omega}=\frac{n}{2\kappa}\left(1-\frac{n+1}{n}\kappa\right)+\varepsilon_2.$$
If $\kappa > \frac{n}{n+1}$, then $M_{\omega} < 0$ by taking a small enough $\epsilon_0>0$. 
As discussed in Remark~\ref{rem:linear}, $\omega_\tau = \omega\circ\tau$ then has a negative growth rate $M_{\omega_\tau}<0$. Since \(M_{f_\tau}>0\) can be chosen
arbitrarily small, we further ensure that $M_{\omega_\tau}+M_{f_\tau} <0$. Applying the first claim in Theorem~\ref{thm:SNIS}, 
RQMC-SNIS attains an $L_p$-error rate of nearly $\mathcal{O}(N^{-1})$ for all $p\ge 1$. If $\kappa = \frac{n}{n+1}$, $M_\omega>0$ can be arbitrarily small and so does $M_{\omega_\tau}$. Since the base Gaussian proposal
satisfies the required sub-Gaussian tail condition by Remark \ref{rem:gauss}, and \(f_\tau\) has finite moments of all orders in this experiment, applying the second claim in Theorem~\ref{thm:SNIS}, 
RQMC-SNIS also attains an $L_p$-error rate of nearly $\mathcal{O}(N^{-1})$ for all $p\ge 1$. 
However, if $\kappa < \frac{n}{n+1}$, $M_{\omega_\tau}>0$ cannot be arbitrarily small, resulting in a severe growth rate for $\omega_\tau$. This is unfavorable for RQMC.

For PriorIS and ODIS with Gaussian proposals, we have a negative growth rate $M_\omega<0$ for $\omega$ by noticing $\lambda_{\max}\bigl(\bm{\Sigma}^{-1}\bigr)=1$. For t proposals, the growth rate $M_{1/q}>0$ can be arbitrarily small, resulting in $M_{\omega} = M_{1/q} +M_{\bar\pi}<0$. Applying the first claim in Theorem~\ref{thm:SNIS} again, 
RQMC-SNIS attains an $L_p$-error rate of nearly $\mathcal{O}(N^{-1})$ for PriorIS, ODIS, tIS0 and tIS.

We first consider the case $d=5$, $\kappa \in \{1/4,3/4,1\}$ and $p \in \{1,2,5\}$. Since the parameter choices for PriorIS, ODIS, and tIS0 are independent of $\kappa$, for $\kappa \in \{3/4,1\}$, we focus only on LapIS and tIS.  Figure \ref{fig:linearG5} shows $\log_2 \widehat{L_p}$ given by \eqref{eq:var} versus $\log_2 N$. 
For PriorIS, ODIS, and tIS0, the observed error decay in the tested range is closer to the MC rate, which is consistent with their small effective sample sizes shown in Figure \ref{linearG2}. In contrast, tIS maintains stable performance across different values of \(p\), and its observed rate is close to the predicted \(\mathcal{O}(N^{-1+\epsilon})\) behavior. LapIS improves as \(\kappa\) increases, but its performance is more sensitive to \(p\) when \(\kappa\) is small.

\begin{figure}[h]
\vspace{-4mm}
\centering
\subfigure[$\kappa = 1/4$.]{\label{linearG1}
        \centering
    \includegraphics[width=0.73\linewidth]{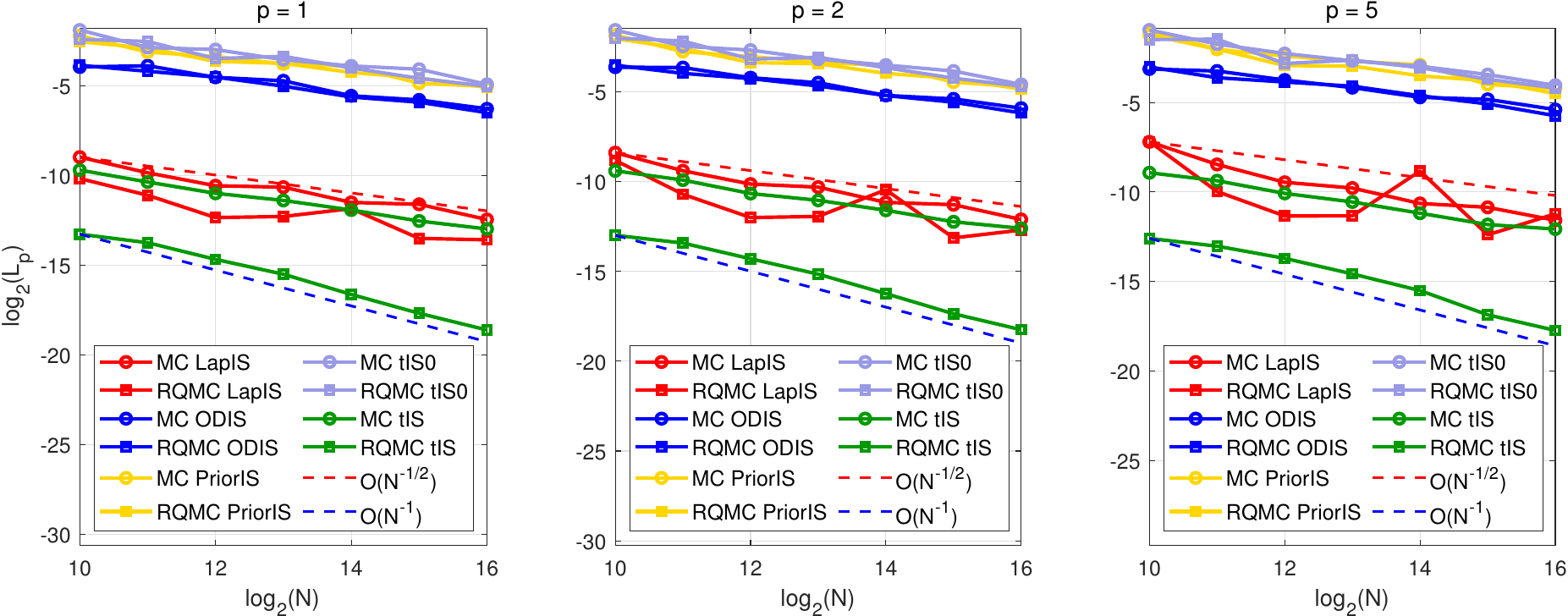}
}
\vspace{-3mm}
\subfigure[Effective sample sizes with $\kappa = 1/4$. ]{\label{linearG2}
        \centering
        \includegraphics[width=0.19\linewidth]{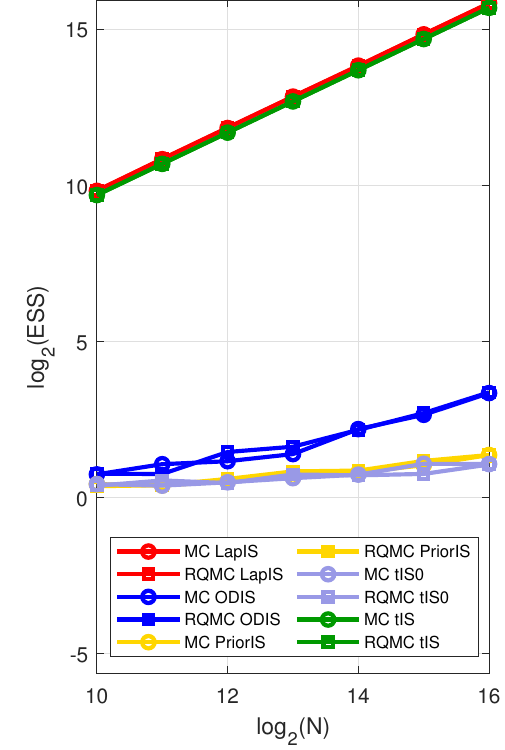}
}
\\
\subfigure[$\kappa = 3/4$.]{\label{linearG3}
        \centering
    \includegraphics[width=0.47\linewidth]{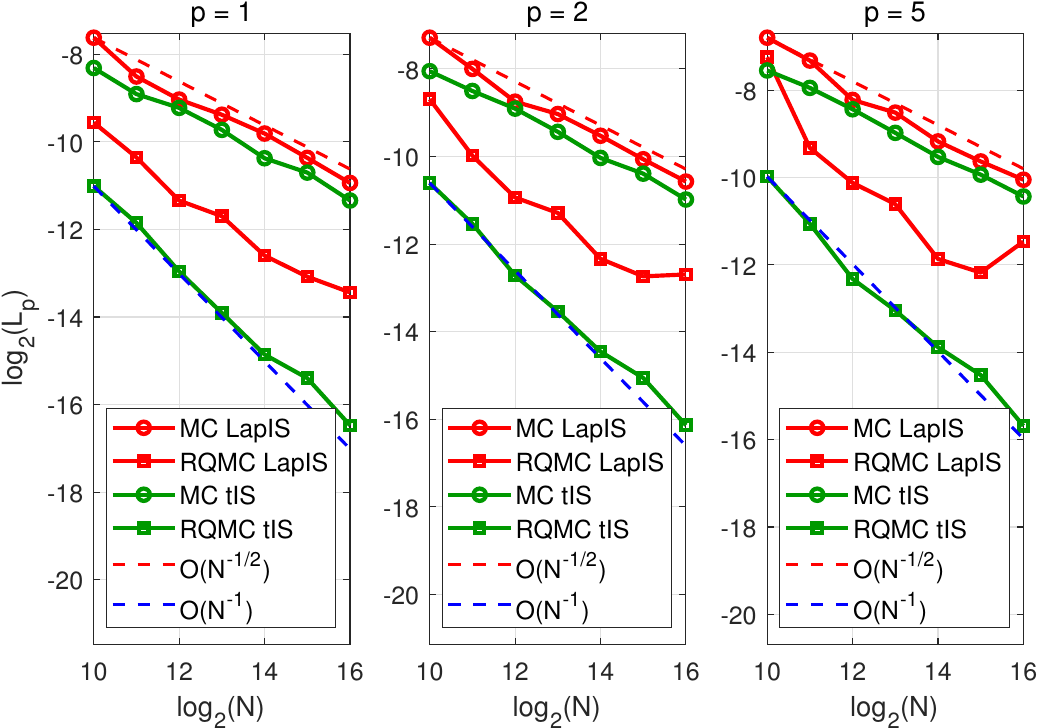}
}
\vspace{-3mm}
\subfigure[$\kappa = 1$.]{\label{linearG4}
        \centering
        \includegraphics[width=0.47\linewidth]{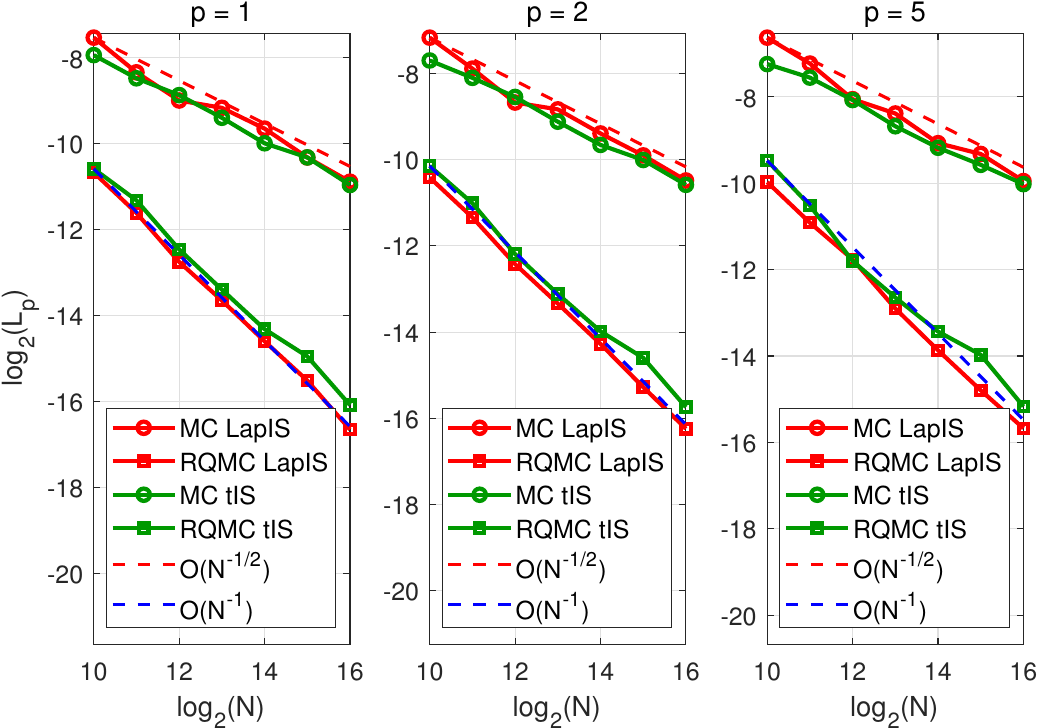}
}
\caption{The $L_p$-error for different proposals with different $\kappa$ and $d = 5$.}
\label{fig:linearG5}
\vspace{-3mm}
\end{figure}

Finally, we consider the case $d = 30$ to assess the effect of dimensionality in Figure~\ref{fig:linearG30}, focusing on $\kappa = 1$. As the dimension increases, the convergence rates of both LapIS and tIS inevitably deteriorate due to the logarithmic factor in the rate. While this factor can be absorbed into $N^{\epsilon}$ for sufficiently large $N$, it still has a noticeable impact when $N$ is relatively small. Developing an RQMC–SNIS algorithm whose performance is dimension-independent remains an open problem. Recently, Pan et al. \cite{pan2025quasi} proposed a boundary-damping importance sampling approach for QMC integration, focusing on component-wise target distributions. Their method can achieve an RMSE rate of order $\mathcal{O}(N^{-1+\epsilon})$ that is independent of the dimension. This work provides a promising direction for addressing the dimensionality challenge in RQMC–SNIS.
\begin{figure}[htbp]
\vspace{-2mm}
\centering
        \centering
        \includegraphics[width=0.6\linewidth]{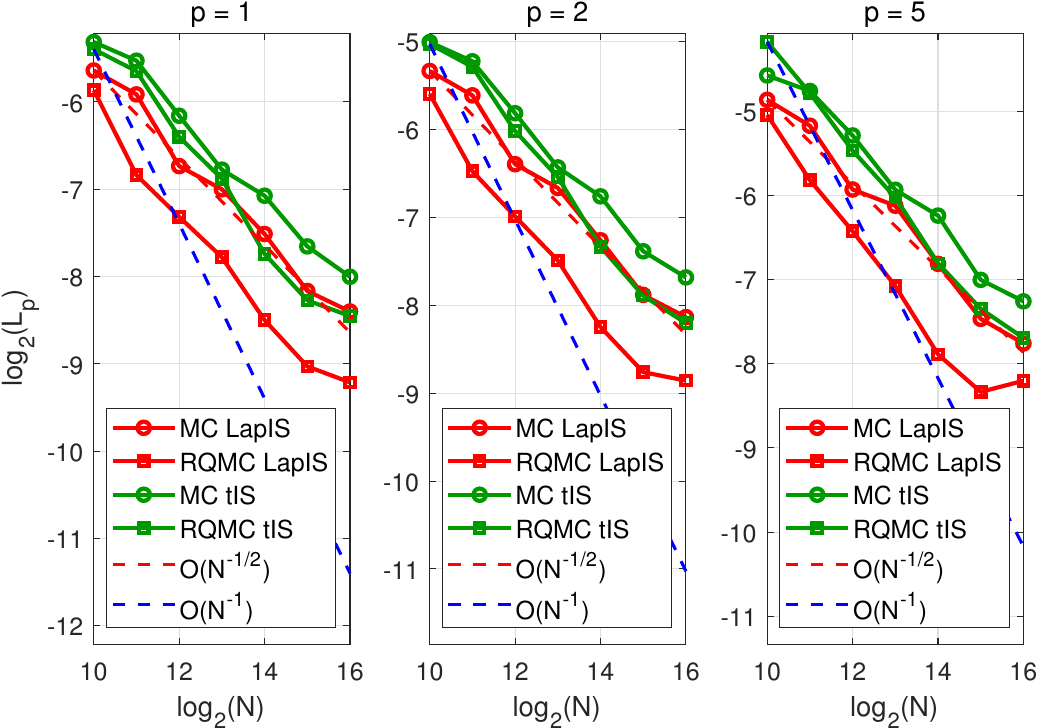}
% }
\caption{The $L_p$-error for different proposals with $\kappa = 1$ and $d = 30$.}
\label{fig:linearG30}
\vspace{-6mm}
\end{figure}

    \section{Conclusions}\label{sec:conclusions}
In this work, we establish the $L_p$-error rate $(p \geq 1)$ of an RQMC-SNIS estimator in the setting of unbounded integrands on $\mathbb{R}^d$ under the
stated assumptions on the target, proposal, integrand, and transport map. We validate our theoretical results through a simple Bayesian inverse problem, examining the impact of various parameters, and find that the experimental outcomes align well with the theory. In addition, we conduct experiments on Bayesian logistic regression, where the results suggest that the linear $t$ proposals constructed using $\bm{\mu}$ and $\bm{\Sigma}$ determined by the LapIS method appear to provide superior performance. Our results fill an important theoretical gap in the foundation of RQMC-SNIS and broaden its applicability to practical Bayesian computation. 

As a by-product, we obtain \(L_p\)-error bounds for plain RQMC integration under critical quadratic growth and admissible transport maps. These bounds are the key technical input for the SNIS analysis and may also be useful in other higher-order error analyses; related applications are discussed in Appendix~\ref{appen:Lp}.

\red{Our results are established under assumptions that
differ from those in \cite{chen2025,deligiannidis2024,ouyang2024}.
In particular, Assumption~\ref{ass:f} requires
$g\in C^d(\mathbb R^d)$ and exponential growth bounds
on all $D^{\bm a}g$ with $|\bm a|\le d$,
which are used in our radial-projection and
transport-map analyses. Moreover,
Theorem~\ref{thm:SNIS}(ii) requires a
sub-Gaussian-type tail condition on the proposal
and derivative growth bounds for the importance weights
with arbitrarily small positive exponents $M_\omega$.
These differences should be taken into account
when comparing the convergence rates with previous results.}

However, in the experiments, we observed that SNIS estimators constructed directly with low-discrepancy point sets still perform poorly in high-dimensional problems. Developing an RQMC-SNIS algorithm that is robust to dimensionality remains an important direction for future research.

\appendix

\section{Auxiliary lemmas and technical proofs}\label{subsec:lemma}
         This appendix provides the proof of Theorem \ref{thm:rqmc} and several auxiliary lemmas. We first consider the case of $M<0$. Define \(h:[0,1]^d\to\mathbb R\) by $h = g \circ \mathcal{T}$. If $\partial^{1:d}h$ does not exist on $\partial[0,1]^d$, the existing bound  \eqref{eq:hk} from \cite{owen2005hk} cannot be applied. Below, we generalize it by removing the requirement on the existence of $\partial^{1:d}h$ on the surface $\partial[0,1]^d$. The $L_p$-error bound then follows from the Koksma--Hlawka inequality.

\begin{lemma}\label{lem:negative_bvhk}
Assume that \(g\) satisfies Assumption~\ref{ass:f} with \(M<0\) and $C_g>0$. Let $h = g \circ \mathcal{T}$,
then 
\begin{shrinkeq}{-1ex}
    \begin{equation}\label{eq:negative_hk}
    V_{\mathrm{HK}}(h)
    \le 
    C_g\left(\uppi/(-M)\right)^{d/2}.
\end{equation}
\end{shrinkeq}
If the RQMC point set satisfies Assumption~\ref{ass:rqmc} with a constant $B_d>0$, then for every \(p\ge1\),
\begin{equation}\label{eq:negative_error}
    \left\{
    \mathbb E_{\mathrm{RQMC}}
    |I_N(g)-q(g)|^p
    \right\}^{1/p}
    \le B_dC_g\left(\uppi/(-M)\right)^{d/2}N^{-1}(\log N)^{d-1}.
\end{equation}
\end{lemma}

\begin{proof}
We first verify the boundary behavior of \(h\). Let
\(\bm v_n\in(0,1)^d\) satisfy $\bm v_n\rightarrow\bm v_\ast
    \in\partial[0,1]^d.$ At least one coordinate of \(\bm v_\ast\) is equal to \(0\) or \(1\). Since \(M<0\), Assumption~\ref{ass:f} gives
\[
\begin{aligned}
    |h(\bm v_n)|
    =
    |g(\mathcal{T}(\bm v_n))|
    \le
    C_g
    \exp\left\{
        M\|\mathcal{T}(\bm v_n)\|^2
    \right\}
    \rightarrow0.
\end{aligned}
\]
Thus \(h\) is continuous on \([0,1]^d\). For \(\bm u\in(0,1)^d\), write $\bm x=\mathcal{T}(\bm u).$ Since \(\mathcal{T}\) acts component-wise, we have
\begin{shrinkeq}{-1ex}
    \[
    \partial^{1:d}h(\bm u)
    =
    \partial^{1:d}g(\bm x)
    \prod_{j=1}^d \frac 1{q_j(x_j)}.
\]
\end{shrinkeq}
Using the change of variables \(x_j=\mathcal{T}_j(u_j)\), we obtain
\begin{equation}\label{eq:negative_vitali}
    \begin{aligned}
    \int_{(0,1)^d}
    |\partial^{1:d}h(\bm u)|\,d\bm u
    &=
    \int_{\mathbb R^d}
    |\partial^{1:d}g(\bm x)|\,d\bm x\\
    &\le
    C_g
    \int_{\mathbb R^d}
    \exp\{M\|\bm x\|^2\}\,d\bm x=
    C_g
    \left(\uppi/(-M)\right)^{d/2}
    <\infty.
\end{aligned}
\end{equation}
Thus $\partial^{1:d}h$ is absolutely integrable in $(0,1)^d.$

We now bound the Vitali variation of \(h\), which is defined in \eqref{eq:vit}. Let $\mathcal Y
    =
    \mathcal Y^1\times\cdots\times\mathcal Y^m$
be an arbitrary multidimensional ladder on \([0,1]^d\). For
\(\bm y\in\mathcal Y\), let $C_{\bm y}=[\bm y,\bm y_+]$ be the corresponding ladder cell. For sufficiently small \(\varepsilon>0\), define $a^{\varepsilon}_j
    =
    \max\{y_j,\varepsilon\}$, $ b^{\varepsilon}_j
    =
    \min\{y_{+,j},1-\varepsilon\}$, and set $C_{\bm y}^{\varepsilon}
    =
    [\bm a^{\varepsilon},
     \bm b^{\varepsilon}].$ 
Clearly, $C_{\bm y}^{\varepsilon}\subset(0,1)^d.$
Since $\partial^{1:d}h$ exists
on \(C_{\bm y}^{\varepsilon}\), Equation~(15) of
Owen~\cite{owen2005hk} gives
\[
    \Delta
    \bigl(
        h;
        \bm a^{\varepsilon},
        \bm b^{\varepsilon}
    \bigr)
    =
    \int_{C_{\bm y}^{\varepsilon}}
    \partial^{1:d}h(\bm u)\,d\bm u.
\]
As \(\varepsilon\downarrow0\), $\bm a^{\varepsilon}\to\bm y$ and $\bm b^{\varepsilon}\to\bm y_+$. 
Since \(h\in C([0,1]^d)\), $\Delta
    \bigl(
        h;
        \bm a^{\varepsilon},
        \bm b^{\varepsilon}
    \bigr)
    \to
    \Delta(h;\bm y,\bm y_+).$
Moreover, $\mathds{1}_{C_{\bm y}^{\varepsilon}}(\bm u)
    \partial^{1:d}h(\bm u)
    \to
    \mathds{1}_{C_{\bm y}}(\bm u)
    \partial^{1:d}h(\bm u)$ for almost every \(\bm u\in(0,1)^d\), and $\left|
    \mathds{1}_{C_{\bm y}^{\varepsilon}}(\bm u)
    \partial^{1:d}h(\bm u)
    \right|
    \le
    \left|
    \partial^{1:d}h(\bm u)
    \right|$. Since \(\partial^{1:d}h\in L^1((0,1)^d)\), the dominated
convergence theorem yields
\[
    \Delta(h;\bm y,\bm y_+)
    =
    \int_{C_{\bm y}}
    \partial^{1:d}h(\bm u)\,d\bm u.
\]
Therefore,
\[
\begin{aligned}
V_{\mathrm{Vit}}(h) &= \sup_{\mathcal{Y}}\sum_{\bm y\in\mathcal Y}
    \left|
        \Delta(h;\bm y,\bm y_+)
    \right| \le \sup_{\mathcal{Y}}\sum_{\bm y\in\mathcal Y}\int_{C_{\bm y}}
    |\partial^{1:d}h(\bm u)|\,d\bm u \\
    &\le
    \int_{[0,1]^d}
    |\partial^{1:d}h(\bm u)|\,d\bm u
    \le
    C_g
    \left(\frac{\uppi}{-M}\right)^{d/2}.
\end{aligned}
\]
where we use the fact that the ladder cells cover \([0,1]^d\) and
have mutually disjoint interiors. 

Now let \(\emptyset\ne\bm v\subsetneq1:d\). Since
\(\bar{\bm v}\ne\emptyset\), the point
\(\bm u_{\bm v}:\bm1_{\bar{\bm v}}\) belongs to
\(\partial[0,1]^d\). Hence $h_{\bm v}(\bm u_{\bm v}) = h(\bm u_{\bm v}:\bm1_{\bar{\bm v}})
    =0$
for every \(\bm u_{\bm v}\in[0,1]^{|\bm v|}\), i.e., $h_{\bm v}\equiv0$. Thus $
V_{\mathrm{Vit}}(h_{\bm v})=0
$. For \(\bm v=1:d\), \(h_{\bm v} = h\). By the definition of the Hardy--Krause variation \eqref{eq:hkvar}, 
$$
V_{\mathrm{HK}}(h)
    =
    \sum_{\emptyset\ne\bm v\subseteq1:d}
    V_{\mathrm{Vit}}(h_{\bm v}) = V_{\mathrm{Vit}}(h).
$$
Combining this identity with \eqref{eq:negative_vitali} proves
\eqref{eq:negative_hk}. Finally, by the Koksma--Hlawka inequality \eqref{eq:KH_inq} with Assumption \ref{ass:rqmc}, we obtain the result \eqref{eq:negative_error}. This completes the proof.
\end{proof}
\red{
\begin{remark}
    Although Assumption~\ref{ass:f} is imposed throughout
the paper, its full $C^d$ regularity is not required
for Lemma~\ref{lem:negative_bvhk}. It suffices that
the first-order mixed partial derivatives
$\partial^{\bm v}g$ are continuous for all
$\bm v\subseteq 1:d$ and satisfy negative
quadratic-exponential growth bounds.
The stronger regularity condition is retained
as a common sufficient assumption for the
radial-projection and transport-map analyses.
\end{remark}
}

        Next we derive the projection-error estimates needed for the radial projection.
        
    \begin{lemma}\label{lem:err1}
      Given $p \geq 1$, assume that $\bm X\sim q\in \mathrm{SG}(\eta,\alpha)$ and that $g$ satisfies Assumption \ref{ass:f} with $0 \le M < \alpha/p$ and a constant $C_g > 0$. If \(A=\alpha/p-M>0\), then there exists a constant $r_0$ such that for any $r>r_0$,
    \begin{equation*}% E2-2: removed number (A.4); original label: error2
    \E_{q}\left|g\circ P_{r}(\X) - g(\X)\right|^{p} \le C_{\mathrm{tail}}r^{\eta}\exp\left(-pA(1-\delta)^2r^2\right),
    \end{equation*}
    where $C_{\mathrm{tail}}$ is a constant independent of \(r\).
    \end{lemma}
    \begin{proof}
    By the definition of \(P_r\) in \eqref{eq:proj}, we have \(P_r(\x)=\x\) whenever
\(\|\x\|\le (1-\delta)r\). Therefore,  
    \[
    \begin{aligned}
        \mathbb{E}_{q}\left|g \circ P_{r}(\X)-g(\X)\right|^p &\le \mathbb{E}_{q}\left[\left(\left|g(\X)\right|+\left|g\left(P_{r}(\X)\right)\right|\right)^p\mathds{1}_{\{\|\X\|>(1-\delta)r\}}\right] \\
        & \le (C_g)^p2^p \mathbb{E}_{q}\left[\exp\left(pM\|\X\|^2\right)\mathds{1}_{\{\|\X\|>(1-\delta)r\}}\right],
    \end{aligned}
    \]
    where Assumption~\ref{ass:f} with \(\bm a=\bm0\) and \(\|P_r(\bm x)\|\le\|\bm x\|\) from Lemma~\ref{lem:pr} were used. It is clear that the upper bound is increasing with respect to $M$. It suffices to examine the case $M>0$ in the following.
Since $q \in \mathrm{SG}(\eta,\alpha)$ as stated in Definition~\ref{ass:x}, there exist constants $t_0>0$  and $C_0>0$ such that for any $t\ge t_0$,
\begin{equation}\label{eq:sg_tail}
    \mathbb P_{q}(\|\bm X\|>t)
    \le
    C_0t^\eta
    \exp\{-\alpha t^2\}.
\end{equation}
Writing $Y := \exp\left(pM\|\X\|^2\right)$ and $y_0 := \exp\left(pM(1-\delta)^2r^2\right)$, we have
    $$
    \begin{aligned}
\E_{q}\bra{\exp\left(pM\|\X\|^2\right)\mathds{1}_{\{\|\X\|>(1-\delta)r\}}}
        &= \int_{0}^{\infty} \PP{}{Y\mathds{1}_{\{Y>y_0\}} > y}dy\\
        &=~   y_0\PP{}{Y > y_0} + \int_{y_0}^{\infty} \PP{}{Y > y}dy.
    \end{aligned}
    $$
Using \eqref{eq:sg_tail}, for any $r> r_1:=t_0(1-\delta)^{-1}$, we have
    $$
    \begin{aligned}
        y_0\PP{}{Y > y_0} &= \exp\left(pM(1-\delta)^2r^2\right)\PP{q}{\|\X\| > (1-\delta)r}\\
        &\le C_0(1-\delta)^{\eta}r^{\eta}\exp\left(-(\alpha-pM)(1-\delta)^2r^2\right).
    \end{aligned}
    $$
    By Lemma \ref{lem:exp}, for any $r> r_2:=(2\alpha-2Mp)^{-1/2}(1-\delta)^{-1}$, we also have
    $$
    \begin{aligned}
        \int_{y_0}^{\infty} \PP{}{Y > y}dy &= \int_{y_0}^{\infty} \PP{q}{\|\X\| > \seq{\frac{\log y}{pM}}^{1/2}}dy\\
        &\le C_0\int_{y_0}^{\infty} \seq{\frac{\log y}{pM}}^{\eta/2}\eexp{-\alpha \seq{\frac{\log y}{pM}}} dy\\
        & \le 2C_{0}(pM)^{-\eta/2} \int_{\sqrt{pM}(1-\delta)r}^{\infty} z^{\eta+1}\eexp{-\frac{\alpha-pM}{pM}z^2} dz \\
        & \le 2C_{0}(1-\delta)^{\eta}(2a)^{-1}C_{\eta+1}r^{\eta}\exp\left(-(\alpha-pM)(1-\delta)^2r^2\right),
    \end{aligned}
    $$
    where $C_s = \max\{(s+2)!!/2,1\}$  and $a=\alpha/(pM)-1$.
    Therefore,
    $$
    \mathbb{E}_{q}\left|g \circ P_{r}(\X)-g(\X)\right|^p \le C_{\mathrm{tail}}r^{\eta}\exp\left(-pA(1-\delta)^2r^2\right)
    $$
    for $C_{\mathrm{tail}} = (C_g)^p2^pC_{0}(1-\delta)^{\eta}(1+C_{\eta+1}/a)$.
    Taking $r_0=\max(r_1,r_2)$ completes the proof.
    \end{proof}

    We now bound the quadrature error for the radially projected integrand.

\begin{lemma}\label{lem:err2}
Assume that $g$ satisfies Assumption 2 with $M\ge 0$ and a constant $C_g>0$. Let $h(\bm u)=(g\circ P_r)\circ \mathcal{T}(\bm u).$ Assume that $\left\{\bm{u}_{1}, \ldots, \bm{u}_{N}\right\}$ is an RQMC point set satisfying Assumption \ref{ass:rqmc} with a constant $B_d > 0$.
    Then, there exists a constant $C_{\mathrm{quad}}> 0$, independent of \(r\), such that
    \[
    \mathbb{E} \left|{I}_{N}(g\circ P_{r})-q(g\circ P_r)\right|^{p} \le C_{\mathrm{quad}} r^{pd}\exp\left(pM(1-\delta/2)^2r^2\right)\frac{(\log N)^{p(d-1)}}{N^p}.
    \]
\end{lemma}

\begin{proof}
    We first bound the Hardy--Krause variation of $h$. 
    By the Hardy--Krause variation in
\eqref{eq:hk}, it suffices to verify that the mixed partial derivatives
\(\partial^{1:d}h\) exists on $[0,1]^d$, and to control
\begin{equation}\label{eq:interhk}
    \sum_{\emptyset\ne\bm v\subseteq1:d}
    \int_{[0,1]^{|\bm v|}}
    \left|
        \partial^{\bm v}h(\bm u_{\bm v}:\bm1_{\bar{\bm v}})
    \right|
    d\bm u_{\bm v}.
\end{equation}
    Let $\x = \mathcal{T}(\bm{u})$. For $\emptyset \neq \bm{v} \subseteq 1:d$, we have
    $$
    \partial^{\bm{v}} h(\bm{u}) = \partial^{\bm{v}} (g\circ P_r)(\x) \prod_{j\in\bm{v}} \frac{d \mathcal{T}_j(u_j)}{d u_j} = \partial^{\bm{v}} (g\circ P_r)(\x) \prod_{j\in\bm{v}} \frac{1}{q_j(\mathcal{T}_j(u_j))}.
    $$
    \red{Assumption~\ref{ass:f} and Lemma~\ref{lem:pr} give $g,P_r\in C^d$,} a special form of the Faa di Bruno formula \cite{basu2016} then gives
    $$ 
    \partial^{\bm{v}} (g\circ P_r)(\x) = \sum_{\boldsymbol{\lambda} \in \mathbb{N}_0^d, 1 \leq|\boldsymbol{\lambda}| \leq|\bm{v}|}  D^{\bm{\lambda}}g(\bm{\vartheta})\sum_{s=1}^{|\bm{v}|} \sum_{{\mathrm{KL}}(s,\bm{v},\bm{\lambda})}\prod_{j=1}^{s} \partial^{\bm{\ell}_j} (P_r(\bm{x}))_{k_j},
    $$
    where $\bm{\vartheta} = P_r(\x)$ and 
    \begin{equation*}% E2-2: removed number (A.7); original label: eq:KL
            {\mathrm{KL}}(s,\bm{v},\bm{\lambda}) = 
        \left\{\begin{array}{r}
            \left(\bm{\ell}_j, k_j\right), j=1, \ldots, s, \mid \red{\emptyset\ne\bm{\ell}_j\subseteq\bm v}, k_j \in 1: d, \cup_{j=1}^{s} \bm{\ell}_j=\bm{v} \\
       \quad \bm{\ell}_j \cap \bm{\ell}_{j^{\prime}} =\emptyset \text { for } j \neq j^{\prime} \text { and }\left|\left\{l \in 1: s \mid k_l=i\right\}\right|= \lambda_i
        \end{array}
        \right\} .
        \end{equation*}
    By Lemma \ref{lem:pr}\ref{pr2}, 
    \(
    |\partial^{\bm{v}} (P_r(\x))_{i}| \le C_P(d,\delta) \mathds{1}_{\{\|\x\| < r\}},
    \) for any $\emptyset \neq \bm{v} \subseteq 1:d$ and $1\le i \le d$.
     Hence, 
    \(
    \left|\prod_{j=1}^{s} \partial^{\bm{\ell}_j} (P_r(\x))_{k_j}\right| \le \left(\prod_{j=1}^s C_P(d,\delta)\right) \mathds{1}_{\{\|\x\| < r\}} \le C_P^{d}(d,\delta)\mathds{1}_{\{\|\x\| < r\}}.
    \)
    Moreover, by Assumption \ref{ass:f} and Lemma \ref{lem:pr}\ref{pr4}, 
    \begin{shrinkeq}{-0ex}
            \[
    |D^{\boldsymbol{\lambda}}g(\bm{\vartheta})| \le C_g \exp(M\|P_r(\x)\|^2) \le C_g \exp\left(M(1-\delta/2)^2r^2\right).
    \]   
    \end{shrinkeq}
     As a result, there exists a constant $C_{d,\delta}>0$ such that 
    \begin{equation}\label{eq:dh}
        |\partial^{\bm{v}} h(\bm{u})| \le C_{d,\delta}C_g\exp\left(M(1-\delta/2)^2r^2\right)\mathds{1}_{\{\|\x\| < r\}}\prod_{j\in\bm{v}} 1/q_j(\mathcal{T}_j(u_j)).
    \end{equation}
    Owing to the identity $\mathds{1}_{\{\|\x\| < r\}} = \mathds{1}_{\{\|\mathcal{T}(\bm{u})\| < r\}}$, the derivatives $\partial^{\bm{v}} h(\bm{u})$ vanish as $u_j \to 0^+$ or $u_j \to 1^-$. Consequently, $\partial^{\bm{v}} h(\bm{u})$ exists on $[0, 1]^d$. Now we control the sum of integrals in \eqref{eq:interhk}. If \(\bar{\bm v}\ne\emptyset\), then at the upper anchor \(\bm u_{\bm v}:\bm1_{\bar{\bm v}}\), at least one coordinate of \(\mathcal{T}(\bm u)\) is infinite, and the indicator \(\mathds{1}_{\{\|\mathcal{T}(\bm u)\|<r\}}\) is zero. Therefore, $|\partial^{\bm{v}} h\left(\uu_{\bm{v}} ; \bm{1}_{\bar{\bm{v}}}\right)| = 0$. For \(\bm v=1:d\), using \eqref{eq:dh} and the change of variables \(x_j=\mathcal{T}_j(u_j)\), we have
    \begin{shrinkeq}{-0.5ex}
        \begin{align*}
     \int_{[0,1]^d}\left|\partial^{1:d} h\left(\uu\right)\right|d\uu 
    \le&~C_{d,\delta}C_g\exp\left(M(1-\delta/2)^2r^2\right)\int_{\|\mathcal{T}(\bm u)\|<r}\frac{1}{q(\mathcal{T}(\bm u))}d\bm u \\
    \le&~C_{d,\delta}C_g\exp\left(M(1-\delta/2)^2r^2\right)\int_{\|\x\|<r}1 d\x \\
    \le &~ C_{d,\delta}C_gV_dr^{d}\exp\left(M(1-\delta/2)^2r^2\right),
\end{align*}
    \end{shrinkeq}
where $V_d$ is the volume of the $d$-dimensional unit ball. Using the Hardy--Krause representation \eqref{eq:hk} completes the proof.
Since the terms with \(\bar v\ne\emptyset\) vanish, the representation
of \(V_{\mathrm{HK}}\) in \eqref{eq:hk} gives
   \begin{equation*}
   \begin{aligned}
       V_{\mathrm{HK}}\left(h\right) &\le \sum_{\emptyset \neq \vv \subseteq 1: d} \int_{[{0}, {1}]^{|\vv|}} \left|\partial^{\bm{v}} h\left(\uu_{\bm{v}} :\bm{1}_{\bar{\bm{v}}}\right)\right|d\uu_{\vv}\\
       &= \int_{[0,1]^{d}} \left|\partial^{1:d} h\left(\uu\right)\right|d\uu ~\le  C_{d,\delta}C_gV_dr^{d}\exp\left(M(1-\delta/2)^2r^2\right).
   \end{aligned}  
\end{equation*}
By the Koksma--Hlawka inequality \eqref{eq:KH_inq} with Assumption \ref{ass:rqmc}, we obtain the result with $C_{\mathrm{quad}} = (B_dC_{d,\delta}C_gV_d)^p$.
\end{proof}

Now we turn to the proof of Theorem \ref{thm:rqmc}.
\begin{proof}[Proof of Theorem \ref{thm:rqmc}]\label{proof}
Part~\(\mathrm{(i)}\) follows directly from
Lemma~\ref{lem:negative_bvhk}. We therefore prove
Part~\(\mathrm{(ii)}\). Assume that
\(0\le M<\alpha/p\) and that
\(\bm X\sim q \in \mathrm{SG}(\eta,\alpha)\). Let \(
    A:=\frac{\alpha}{p}-M>0.\)
Combining \eqref{eq:err0}, \eqref{eq:err1}, and \eqref{eq:err2} with the results of Lemmas \ref{lem:err1} and \ref{lem:err2}, and using $(a+b+c)^p \le 3^{p-1}(a^p+b^p+c^p)$ with $a,b,c\ge0$, we obtain, there exists a constant $r_0$ such that for $r > r_0$, , $\mathbb{E}\left|{I}_{N}(g)-q(g)\right|^{p}$ is bounded by
    \begin{equation}\label{eq:errC}
        C_0\left(r^{\eta}\exp\left(-pA(1-\delta)^2r^2\right) + r^{pd}\exp\left(pM(1-\delta/2)^2r^2\right)\frac{(\log N)^{p(d-1)}}{N^p}\right),
    \end{equation}
    where $C_0:=3^{p-1}(2C_{\rm tail}+C_{\rm quad})$, $C_{\rm tail}$ and $C_{\rm quad}$ are given in Lemmas \ref{lem:err1} and \ref{lem:err2}, respectively.
The projection radius \(r\) is a free parameter. The first term in
\eqref{eq:errC} is the truncation error and decreases exponentially in \(r^2\),
whereas the second term is the QMC quadrature error for the projected function
and increases exponentially in \(r^2\). We therefore choose \(r\) on the
logarithmic scale $r = \sqrt{\theta\log N}$, where \(\theta>0\) is a fixed tuning parameter independent of \(N\). Substituting $r = \sqrt{\theta\log N}$ into \eqref{eq:errC} gives
\begin{equation*}% E2-2: removed number (A.10); original label: eq:errc1
    C_1\Big[
(\log N)^{\eta/2}
N^{-pA(1-\delta)^2\theta}
+
(\log N)^{p(3d/2-1)}
N^{-p+pM(1-\delta/2)^2\theta}
\Big],
\end{equation*}
where $C_1
    :=
    C_0\max\{\theta^{\eta/2},\theta^{pd/2}\}$.
To make the two polynomial decay rates $N^{-pA(1-\delta)^2\theta}$ and $N^{-p+pM(1-\delta/2)^2\theta}$ identical, we equate the exponents by choosing
\[
\theta = \theta^* = \frac{1}{(1-\delta)^2A+(1-\delta/2)^2M}.
\]
This choice yields a unified polynomial decay rate of $N^{-p\gamma_\delta}$, where 
\[
\gamma_\delta = A(1-\delta)^2\theta^* = \frac{A}{A+\left(1+\frac{\delta}{2(1-\delta)}\right)^2M}.
\]
Hence, by setting $r=\sqrt{\theta^*\log N}$, we obtain
    \begin{equation}\label{eq:errt}
        \mathbb{E}\left|{I}_{N}(g)-q(g)\right|^{p} \le C_1 N^{-p\gamma_\delta} (\log N)^{\max \{\eta/2,p(3d/2-1)\}}.
    \end{equation}
Now fix an arbitrarily small \(\epsilon>0\). Choose
\(\epsilon_1,\epsilon_2>0\) such that
\(
    \epsilon_1+\epsilon_2<\epsilon.
\) 
Let $\eta* = \max\{\eta/2,p(3d/2-1)\}$, $C_{\log} = \sup_{N\ge1}(\log N)^{\eta*}N^{-p\epsilon_1}$, we have 
\(
    (\log N)^{\eta*} \le C_{\log}N^{p\epsilon_1}.
\)
Furthermore, since
\[
    \gamma_\delta
    =
    \frac{A}
    {A+\left(1+\frac{\delta}{2(1-\delta)}\right)^2M}
    \to
    \frac{A}{A+M}
    =
    1-\frac{pM}{\alpha}
    \qquad \text{as } \delta\rightarrow 0,
\]
we may choose \(\delta\in(0,1/2)\)
sufficiently small so that
\(
    \gamma_\delta
    \ge
    1-\frac{pM}{\alpha}-\epsilon_2.
\)
Applying these two estimates to \eqref{eq:errt}, yields
\[
\mathbb{E}\left|I_N(g)-q(g)\right|^p
\le
C_{\mathrm{main}}(d)
N^{p(-\gamma+\epsilon)},
\]
where $\gamma = 1-pM/\alpha$ and $C_{\mathrm{main}}(d) = C_1C_{\log}.$
This completes the proof.
\end{proof}

\begin{remark}\label{rem:necessity}
    Lemma~\ref{lem:negative_bvhk} shows that, if \(M<0\), then
    \(g\circ \mathcal{T}\) is already of bounded variation in the sense of Hardy and
    Krause. Hence no projection operator is needed in the negative-growth case. However, the radial projection is needed in the critical case \(M>0\). If the component-wise projection operator \(P_r^{\rm cw}\) was used, then
    \(\|P_r^{\rm cw}(x)\|^2\le d r^2\), and the Hardy--Krause bound in Lemma \ref{lem:err2} would contain
    the factor \(\exp(Mdr^2)\). Balancing the resulting projection and quadrature
    terms with \(r=\sqrt{\theta\log N}\) gives the convergence rate nearly $\mathcal{O}(N^{-p\gamma_{\rm{cw}}})$, 
    where $\gamma_{\rm{cw}} = \frac{A}{A+Md}$ and $A = \alpha/p - M$.
    Crucially, $\gamma_{\rm{cw}}$ deteriorates as \(d\) increases. 
    In contrast, our radial operator ensures $\|P_r(x)\|^2 \le (1-\delta/2)^2r^2$
     independently of $d$, preserving a dimension-independent convergence index
      $\gamma = \frac{A}{A+M} = 1-pM/\alpha$. 
\end{remark}

\section{Bayesian logistic regression}\label{app:logistic}
We consider a Bayesian logistic regression model following \cite{he2023}. Let $\bm{X} \in \mathbb{R}^{m\times d}$ denote the observation matrix of predictor variables with $\bm{X}_i\in \mathbb{R}^{1\times d}$ the $i$-th row, $\bm{\beta} =(\beta_1,\beta_2,\ldots,\beta_d)^{T}\in\mathbb{R}^{d\times 1}$ denote the regression parameters, and \( \bm{Y} = (Y_1,Y_2,\ldots,Y_m) \) denote the observation with $Y_i\in \{0,1\}$. The logistic model is expressed as $$\SPP{Y_i = 1} =1-\SPP{Y_i =0}= \frac {\exp(\bm{X}_i\bm{\beta})}{1+\exp(\bm{X}_i\bm{\beta})},\ i=1,2,\ldots,m.$$ The log-likelihood function for \( \bm{Y} \) is given by
$$
\log l(\bm{\beta}) = \bm{\beta}^T \bm{X}^T \bm{Y}-\sum_{i=1}^m \log \left(1+\exp \left( \bm{X}_i\bm{\beta}\right)\right).
$$
We take a standard Gaussian prior $\bm{\beta} \sim \mathcal{N}(\bm{0}, \bm{I}_d)$. We consider the Pima Indian dataset with $m = 392, d=9$ \cite{windle2013}. Following \cite{he2023}, we compare two proposal distributions:
\begin{itemize}
    \item LapIS: A Gaussian proposal with mean $ \boldsymbol{\mu}_{\star}$ and covariance $\bm{\Sigma}_{\star} $, where $\boldsymbol{\mu}_{\star}$ solves
$$\boldsymbol{\mu}_{\star} = \sum_{i=1}^m\left(Y_i-\frac{\exp \left(\X_i \boldsymbol{\mu}_{\star}\right)}{1+\exp \left(\X_i \boldsymbol{\mu}_{\star}\right)}\right) \X_i^T, $$
and the covariance $\bm{\Sigma}_{\star} $
% \red{($\Sigma$ is not a bold text before, unify it)}
is given by
$$
\boldsymbol{\Sigma}_{\star}=\left(I_d-\nabla^2 F\left(\boldsymbol{\mu}_{\star}\right)\right)^{-1},$$
where $F(\bm \beta)=\log l(\bm{\beta})$ and
$$
\nabla^2 F\left(\boldsymbol{\mu}_{\star}\right)_{j k}=-\sum_{i=1}^m \frac{\X_{i j} \X_{i k} \exp \left(\X_i \boldsymbol{\mu}_{\star}\right)}{\left(1+\exp \left(\X_i \boldsymbol{\mu}_{\star}\right)\right)^2} .
$$
    \item Linear t proposal, $\bm{\mu} = \boldsymbol{\mu}_{\star},\bm{\Sigma} = \bm{\Sigma}_{\star}$. 
\end{itemize}

By Faa di bruno formula \cite{cons1996}, we can check that the likelihood function $l(\bm{\beta})$ possesses an arbitrarily small growth rate $M_l>0$. Hence, $M_{\bar \pi} = -1/2+M_l$. Then, using the fact that $\lambda_{\max}(\bm{\Sigma}_{\star}^{-1})>1$ by Theorem 5.3 of \cite{he2023}, we have $M_{\omega}=M_{1/q}+M_{\bar \pi}>\frac 12 \lambda_{\max}(\bm{\Sigma}_{\star}^{-1})+ M_{\bar \pi}> 0$. This is an unfavorable case for achieving a fast convergence rate of RQMC.
% for any $\varepsilon_1 > 0$ and $\varepsilon_2 > 0$, we have
% $$
% \alpha^{\mathrm{LapIS}} = \lambda_{\min}(\bm{\Sigma}_{\star}^{-1})/2-\varepsilon_1, \quad M^{\mathrm{LapIS}}_{\omega f,\tau} 
% = \frac{1}{2}\left(\lambda_{\max}(\bm{\Sigma}_{\star}^{-1})-1\right)\lambda_{\max}(\bm{\Sigma}_{\star})(1+\varepsilon_2) > 0,
% $$
% where we have use the fact that $\lambda_{\max}(\bm{\Sigma}_{\star}^{-1})>1$ by Theorem 5.3 of \cite{he2023}.
% However, we cannot guarantee that $\alpha^{\mathrm{LapIS}}/M^{\mathrm{LapIS}}_{\omega f,\tau}>1$,
% and consequently our theoretical results do not apply to LapIS in this model. On the other hand, for tIS we have $M_{\omega f}^{\mathrm{tIS}}<0$ and $M^{\mathrm{tIS}}_{\omega f,\tau}<0$. 
By contrast, the tIS proposal achieves an $L_p$-error rate of nearly $\mathcal{O}(N^{-1})$. Figure \ref{fig:log} illustrates the estimate of $L_p$-error for $p = \{1,2,5\}$. It can be observed that LapIS performs worse than tIS, whereas the convergence rates of tIS remain consistently close to $\mathcal{O}(N^{-1})$ across different values of $p$.
\begin{figure}[htbp]
        \vspace{-4mm}
        \centering
        \includegraphics[width=0.6\linewidth]{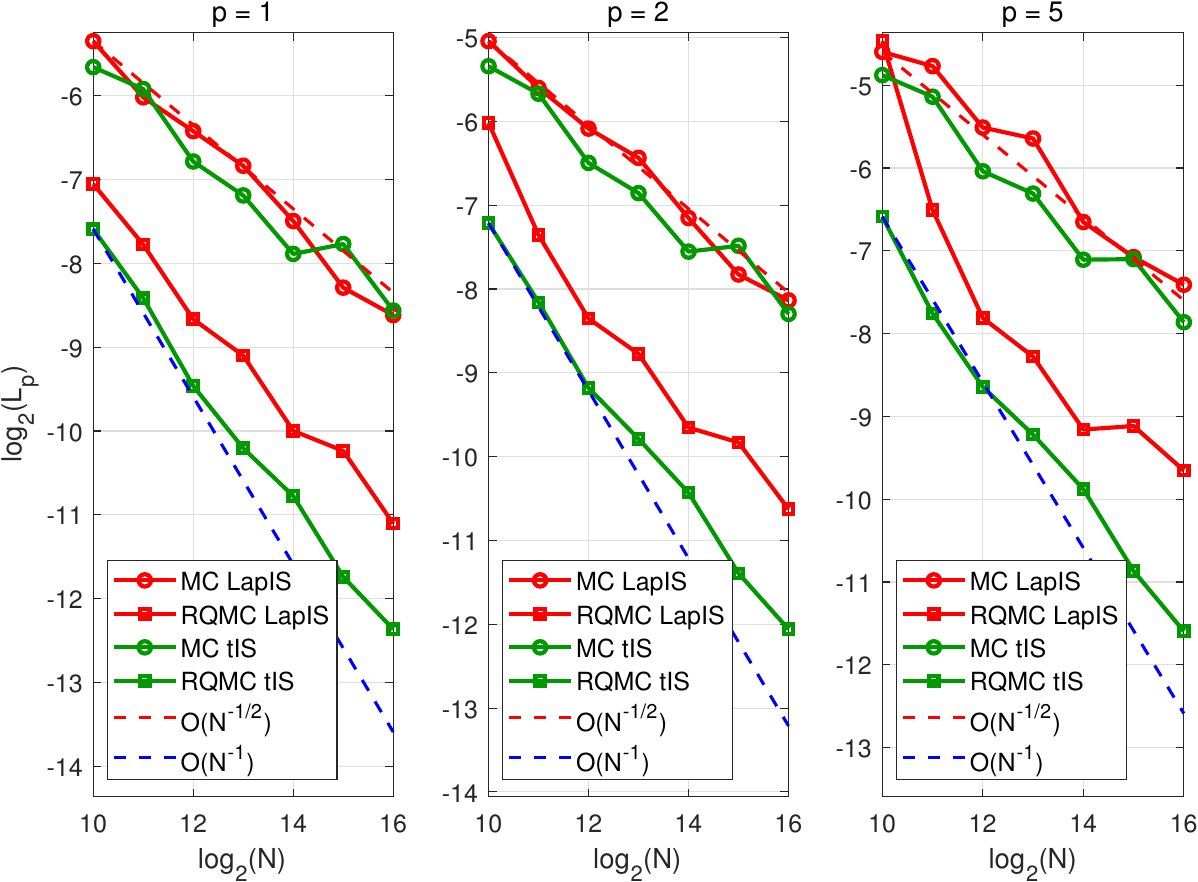}
        \caption{The $L_p$-error for different proposals with Pima dataset.}
        \label{fig:log}
        \vspace{-4mm}
        \end{figure}

        \section{Some applications of $L_p$-error for RQMC}\label{appen:Lp}
% Here, we discuss some applications of $L_p$-error for RQMC integration.
        % \noindent
        
        % \textbf{Hölder’s inequality.}
        % The $L_p$-error rate for RQMC can be applied via Hölder’s inequality to analyze the error rate of the RQMC based adaptive multiple importance sampling method \cite{chen2025}.
        
        % \textbf{Higher-order Markov inequality.}
        % The $L_p$-error rate for RQMC can also be used through higher-order Markov inequalities. This serves as the key tool in analyzing the error rate of the RQMC-based SNIS method in Section \ref{sec:QMCSNIS}.
        
%        \textbf{Skewness and kurtosis.} 
        Conveniently, the expectation $\E$ in this section is taken over the RQMC randomization.
        Using the $L_p$-error for RQMC integration, we can bound the skewness $\gamma$ and the kurtosis $\kappa$ in RQMC, which are important quantities in some properties of confidence intervals. For convenience, we focus on the case $f$ is ``QMC-friendly", which means that for any $\epsilon>0$ and $p\ge 1$, $
        \set{\E\left|{I}_{N}(f)-\pi(f)\right|^{p}}^{1/p} = \mathcal{O}\left(N^{-1+\epsilon}\right).
        $
        Assume that $\E[({I}_{N}(f)-\pi(f))^2] = \Omega(N^{-2})$. Then
        \begin{equation}\label{eq:ske}
             |\gamma| = \frac{|\E[({I}_{N}(f)-\pi(f))^3]|}{\E[({I}_{N}(f)-\pi(f))^2]^{3/2}} \le \frac{\E[|{I}_{N}(f)-\pi(f)|^3]}{\E[({I}_{N}(f)-\pi(f))^2]^{3/2}} = \frac{\mathcal{O}\left(N^{-3+\epsilon}\right)}{\Omega\left(N^{-3}\right)}= \mathcal{O}\left(N^{\epsilon}\right),
        \end{equation} 
        and,
        \begin{equation}\label{eq:kur}
             \kappa = \frac{\E[({I}_{N}(f)-\pi(f))^4]}{\E[({I}_{N}(f)-\pi(f))^2]^{2}} = \frac{\mathcal{O}\left(N^{-4+\epsilon}\right)}{\Omega\left(N^{-4}\right)} = \mathcal{O}\left(N^{\epsilon}\right).
        \end{equation}  
        
        Based on these, we can further analyze some properties of confidence intervals for RQMC. Recently, it has been observed that a standard Student's t based confidence interval is empirically more effective than bootstrap t interval \cite{l2023confidence}. Pan and Owen \cite{pan2025skewness} explained the empirical phenomenon by analyzing the skewness of RQMC. For a review of confidence intervals in the RQMC setting, see \cite{owen2024error}.
        
        In RQMC, we independently repeat the estimator $R$ times, yielding $R$ IID random variables $X_i = {I}_{N}^{(i)}(f)$. For $0 < a < 1/2$, the standard Student's t based confidence interval for $\pi(f)$ with nominal coverage $1-2 a$ takes the form
        $$
        \bar{X} \pm t_{(R-1)}^{1-a} S / \sqrt{R},
        $$
        where
        $$
        \bar{X}=\frac{1}{R} \sum_{i=1}^R X_i, \quad S^2=\frac{1}{R-1} \sum_{i=1}^R\left(X_i-\bar{X}\right)^2,
        $$
        and $t_{(R-1)}^{1-a}$ is the $1-a$ quantile of the Student's $t$ distribution on $R-1$ degrees of freedom.
        
%        \textbf{Berry-Esseen bound}

        Pan and Owen \cite{pan2025skewness} considered the skewness for RQMC directly without an absolute value in the numerator. As a consequence, the result of \cite{pan2025skewness} cannot be applied to the classical Berry–Esseen bound \cite{bentkus1996berry}. In contrast, our estimation of skewness is based on the absolute third moment, which ensures that the Berry–Esseen bound can be applied. 
        Define the Student’s t statistic $T_R = \frac{\bar{X}-\pi(f)}{S / \sqrt{R}}$. A Berry-Esseen type bound for $T_R$ quantifies the worst-case deviation of its distribution to the standard normal distribution, given by $
        \Gamma_R =  \sup\limits_{x\in \mathbb{R}}| \SPP{T_R<x}-\Phi(x)|.$
        By Theorem 1.1 of \cite{bentkus1996berry}, there exists an absolute constant $c>0$ such that 
        $$
        \sup\limits_{x\in \mathbb{R}}\abs{\SPP{\frac{\bar{X}-\pi(f)}{\hat\sigma/\sqrt{R}}<x}-\Phi(x)} \le \frac{c}{\sqrt{R}}
        \frac{\E[|{I}_{N}(f)-\pi(f)|^3]}{\E[({I}_{N}(f)-\pi(f))^2]^{3/2}},
        $$
        where $\hat\sigma^2 = \frac{R-1}{R}S^2$. Hence, 
        \begin{align*}
             \Gamma_R &= \sup\limits_x
            \abs{\SPP{\frac{\bar{X}-\pi(f)}{\hat\sigma/\sqrt{R}}<\sqrt{\frac{R-1}{R}}x}-\Phi(x)}\\
            & \le \sup\limits_x\abs{ \SPP{\frac{\bar{X}-\pi(f)}{\hat\sigma/\sqrt{R}}<\sqrt{\frac{R-1}{R}}x}-\Phi\seq{\sqrt{\frac{R-1}{R}}x} }\\&\quad +\abs{\Phi\seq{\sqrt{\frac{R-1}{R}}x}-\Phi(x)}\\
            &\le \frac{c}{\sqrt{R}}
            \frac{\E[|{I}_{N}(f)-\pi(f)|^3]}{\E[({I}_{N}(f)-\pi(f))^2]^{3/2}}+\mathcal{O}(R^{-1}) = \mathcal{O}(N^{\epsilon})\mathcal{O}(R^{-1/2})+\mathcal{O}(R^{-1}).
        \end{align*}
        Since $\epsilon>0$ is arbitrarily small, $\mathcal{O}(N^{\epsilon})$ is almost $\mathcal{O}(1)$. Then, $\Gamma_R \approx \mathcal{O}(R^{-1/2})$.
        
%        \textbf{Coverage error in Student's t interval}
  
        Another important quantity is the coverage error in the standard confidence interval, defined as
        $$
        \kappa_R = \SPP{\bar{X}-\frac{S}{\sqrt{R}} t_{(R-1)}^{1-a} \leqslant \pi(f) \leqslant \bar{X}+\frac{S}{\sqrt{R}} t_{(R-1)}^{1-a}}-(1-2a).
        $$
        By  \cite{owen2025coverage}, the coverage error in Student's t interval is given by 
        $$
        \begin{aligned}
        \frac{z_{1-a}}{R}\left(\frac{z_a^2-3}{6} \kappa-\frac{z_a^4+2 z_a^2-3}{9} \gamma^2\right) \Phi'\left(z_a\right)+\mathcal{O}\left(R^{-3 / 2}\right).
        \end{aligned}
        $$
        where $z_a = \Phi^{-1}(a)$ and $\Phi'(x)$ is the standard normal density. By \eqref{eq:ske} and \eqref{eq:kur}, we have $\kappa_R = \mathcal{O}(N^{2\epsilon})\mathcal{O}(R^{-1})+\mathcal{O}(R^{-3/2}).$
        Since $\epsilon$ can be arbitrarily small, the convergence rate of $\kappa_R$ is nearly $\mathcal{O}(R^{-1})$.

\section*{Acknowledgments}
We would like to acknowledge the Associate Editor and two anonymous referees for helpful comments to improve the quality of the paper.

\IfFileExists{arxiv/myreferences.bib}{
\bibliographystyle{arxiv/siamplain}
\bibliography{arxiv/myreferences}
}{
\bibliographystyle{siamplain}
\bibliography{myreferences}

\begin{thebibliography}{10}

\bibitem{agapiou2017}
{\sc S.~Agapiou, O.~Papaspiliopoulos, D.~Sanz-Alonso, and A.~M. Stuart}, {\em
  {Importance sampling: Intrinsic dimension and computational cost}},
  Statistical Science,  (2017), pp.~405--431.

\bibitem{basu2016}
{\sc K.~Basu and A.~B. Owen}, {\em {Transformations and Hardy--Krause
  variation}}, SIAM Journal on Numerical Analysis, 54 (2016), pp.~1946--1966.

\bibitem{bentkus1996berry}
{\sc V.~Bentkus and F.~G{\"o}tze}, {\em {The Berry-Esseen bound for Student's
  statistic}}, The Annals of Probability, 24 (1996), pp.~491--503.

\bibitem{chen2025}
\red{{\sc J.~Chen, J.~Du, X.~Wang, and Z.~He}, {\em {Enhanced convergence rates
  of adaptive importance sampling with recycling schemes via quasi-Monte Carlo
  methods}}, Statistics and Computing, 36, 178 (2026).}

\bibitem{cons1996}
{\sc G.~M. Constantine and T.~H. Savits}, {\em {A Multivariate Faa di Bruno
  Formula with Applications}}, Transactions of the American Mathematical
  Society, 348 (1996), pp.~503--520.

\bibitem{deligiannidis2024}
\red{{\sc G.~Deligiannidis, P.~E. Jacob, E.~M. Khribch, and G.~Wang}, {\em {On
  importance sampling and independent Metropolis-Hastings with an unbounded
  weight function}}, The Annals of Statistics,  (2026).
\newblock To appear.}

\bibitem{bookdick2022}
{\sc J.~Dick, P.~Kritzer, and F.~Pillichshammer}, {\em {Lattice Rules:
  Numerical Integration, Approximation, and Discrepancy}}, Springer Nature,
  2022.

\bibitem{bookdick2010}
{\sc J.~Dick and F.~Pillichshammer}, {\em {Digital Nets and Sequences:
  Discrepancy Theory and Quasi-Monte Carlo Integration}}, Cambridge University
  Press, 2010.

\bibitem{dick2019}
{\sc J.~Dick, D.~Rudolf, and H.~Zhu}, {\em {A weighted discrepancy bound of
  quasi-Monte Carlo importance sampling}}, Statistics \& Probability Letters,
  149 (2019), pp.~100--106.

\bibitem{ferger2014}
{\sc D.~Ferger}, {\em {Optimal constants in the Marcinkiewicz--Zygmund
  inequalities}}, Statistics \& Probability Letters, 84 (2014), pp.~96--101.

\bibitem{sqmc2015}
{\sc M.~Gerber and N.~Chopin}, {\em {Sequential quasi Monte Carlo}}, Journal of
  the Royal Statistical Society Series B: Statistical Methodology, 77 (2015),
  pp.~509--579.

\bibitem{glasserman2004}
{\sc P.~Glasserman}, {\em {Monte Carlo Methods in Financial Engineering}},
  vol.~53, Springer, 2004.

\bibitem{he2024}
{\sc Z.~He, H.~Wang, and X.~Wang}, {\em {Quasi-Monte Carlo and importance
  sampling methods for Bayesian inverse problems}}, arXiv preprint
  arXiv:2403.11374,  (2024).

\bibitem{he2023}
{\sc Z.~He, Z.~Zheng, and X.~Wang}, {\em {On the error rate of importance
  sampling with randomized quasi-Monte Carlo}}, SIAM Journal on Numerical
  Analysis, 61 (2023), pp.~515--538.

\bibitem{joe2003}
{\sc S.~Joe and F.~Y. Kuo}, {\em {Remark on algorithm 659: Implementing Sobol's
  quasirandom sequence generator}}, ACM Trans. Math. Softw. (TOMS), 29 (2003),
  pp.~49--57.

\bibitem{liu2024transport}
\red{{\sc S.~Liu}, {\em {Transport quasi-Monte Carlo}}, SIAM Journal on
  Scientific Computing, 48 (2026), pp.~C658--C683.}

\bibitem{le2009}
{\sc P.~L’Ecuyer}, {\em {Quasi-Monte Carlo methods with applications in
  finance}}, Finance and Stochastics, 13 (2009), pp.~307--349.

\bibitem{l2023confidence}
{\sc P.~l’Ecuyer, M.~K. Nakayama, A.~B. Owen, and B.~Tuffin}, {\em
  {Confidence intervals for randomized quasi-Monte Carlo estimators}}, in 2023
  Winter Simulation Conference (WSC), IEEE, 2023, pp.~445--456.

\bibitem{booknieder1992}
{\sc H.~Niederreiter}, {\em {Random Number Generation and Quasi-Monte Carlo
  Methods}}, Society for Industrial and Applied Mathematics, 1992.

\bibitem{ouyang2024}
{\sc D.~Ouyang, X.~Wang, and Z.~He}, {\em {Achieving high convergence rates by
  quasi-Monte Carlo and importance sampling for unbounded integrands}}, SIAM
  Journal on Numerical Analysis, 62 (2024), pp.~2393--2414.

\bibitem{owen1995randomly}
{\sc A.~B. Owen}, {\em {Randomly permuted (t, m, s)-nets and (t,
  s)-sequences}}, in Monte Carlo and Quasi-Monte Carlo Methods in Scientific
  Computing: Proceedings of a conference at the University of Nevada, Las
  Vegas, Nevada, USA, June 23--25, 1994, Springer, 1995, pp.~299--317.

\bibitem{owen2005hk}
{\sc A.~B. Owen}, {\em {Multidimensional variation for quasi-Monte Carlo}}, in
  Contemporary Multivariate Analysis And Design Of Experiments: In Celebration
  of Professor Kai-Tai Fang's 65th Birthday, 2005, pp.~49--74.

\bibitem{owen2006}
{\sc A.~B. Owen}, {\em {Halton sequences avoid the origin}}, SIAM review, 48
  (2006), pp.~487--503.

\bibitem{owen2025coverage}
{\sc A.~B. Owen}, {\em {Coverage errors for Student's t confidence intervals
  comparable to those in Hall (1988)}}, arXiv preprint arXiv:2501.07645,
  (2025).

\bibitem{owen2024error}
{\sc A.~B. Owen}, {\em {Error estimation for quasi-Monte Carlo}}, in Monte
  Carlo and Quasi-Monte Carlo 2024, C.~Lemieux and B.~Feng, eds., vol.~522 of
  Springer Proceedings in Mathematics \& Statistics, Springer, Cham, 2026,
  pp.~95--125.

\bibitem{pan2025quasi}
{\sc Z.~Pan, D.~Ouyang, and Z.~He}, {\em {Quasi-Monte Carlo integration over
  $\mathbb{R}^s$ with boundary-damping importance sampling}}, arXiv preprint
  arXiv:2509.07509,  (2025).

\bibitem{pan2025skewness}
{\sc Z.~Pan and A.~B. Owen}, {\em {Skewness of a randomized quasi-Monte Carlo
  estimate}}, Journal of Complexity,  (2025), p.~101956.

\bibitem{sanz2023}
{\sc D.~Sanz-Alonso, A.~Stuart, and A.~Taeb}, {\em {Inverse problems and data
  assimilation}}, vol.~107, Cambridge University Press, 2023.

\bibitem{vershynin2018}
{\sc R.~Vershynin}, {\em {High-Dimensional Probability: An Introduction with
  Applications in Data Science}}, Cambridge University Press, 2018.

\bibitem{wang2017}
{\sc C.~Weng, X.~Wang, and Z.~He}, {\em {Efficient computation of option prices
  and Greeks by quasi--Monte Carlo method with smoothing and dimension
  reduction}}, SIAM Journal on Scientific Computing, 39 (2017), pp.~B298--B322.

\bibitem{windle2013}
{\sc J.~B. Windle}, {\em {Forecasting high-dimensional, time-varying
  variance-covariance matrices with high-frequency data and sampling
  P\'olya-Gamma random variates for posterior distributions derived from
  logistic likelihoods}}, PhD thesis, The University of Texas at Austin, 2013.

\bibitem{Zhang2021}
{\sc C.~Zhang, X.~Wang, and Z.~He}, {\em {Efficient importance sampling in
  quasi-Monte Carlo methods for computational finance}}, SIAM Journal on
  Scientific Computing, 43 (2021), pp.~B1--B29.

\end{thebibliography}
}
\end{sloppypar}
\end{document}